\documentclass[12pt,a4wide,english]{article}

\usepackage{color}
\usepackage{amsfonts}
\usepackage{latexsym,color,lscape}
\RequirePackage{amsthm}
\RequirePackage[cmex10]{amsmath}
\usepackage{amssymb}
\usepackage{natbib}
\usepackage{graphicx}
\usepackage{verbatim}

  \theoremstyle{plain}
  \newtheorem{lem}{\protect\lemmaname}
  \theoremstyle{remark}
  
  \theoremstyle{definition}
  \newtheorem{defn}{\protect\definitionname}
  \theoremstyle{plain}
  \newtheorem{assumption}{\protect\assumptionname}
\theoremstyle{plain}
\newtheorem{thm}{\protect\theoremname}
  \theoremstyle{plain}
  \newtheorem{cor}{\protect\corollaryname}
  \theoremstyle{plain}
  \newtheorem{prop}{\protect\propositionname}

\makeatother

\def\Pr{\mathbb{P}}
  \providecommand{\assumptionname}{Assumption}
  \providecommand{\definitionname}{Definition}
  \providecommand{\lemmaname}{Lemma}
  \providecommand{\propositionname}{Proposition}
  \providecommand{\remarkname}{Remark}
\providecommand{\corollaryname}{Corollary}
\providecommand{\theoremname}{Theorem}

\newcommand{\cvd}{\mbox{$\stackrel{d}{\longrightarrow}\,$}}

\newcommand{\cvp}{\mbox{$\stackrel{p}{\longrightarrow}\,$}}
\def\ind{\mathbb{I}}

\begin{document}

\title{A Significance Test for Covariates
\\
in Nonparametric Regression}

\author{Pascal Lavergne\\
Toulouse School of Economics
\\
Samuel Maistre and
Valentin Patilea\\
Crest-Ensai \& Irmar (UEB)
}
\date{March 2014}
\maketitle

\begin{abstract}
We consider testing the significance of a subset of covariates in a
nonparametric regression.  These covariates can be continuous and/or
discrete.  We propose a new kernel-based test that smoothes only over
the covariates appearing under the null hypothesis, so that the curse
of dimensionality is mitigated.  The test statistic is asymptotically
pivotal and the rate of which the test detects local alternatives
depends only on the dimension of the covariates under the null
hypothesis.  We show the validity of wild bootstrap for the test. In
small samples, our test is competitive compared to existing
procedures.
\end{abstract}

\newpage
\section{Introduction}

 Testing the significance of covariates is common in applied
 regression analysis. Sound parametric inference hinges on the correct
 functional specification of the regression function, but the
 likelihood of misspecification in a parametric framework cannot be
 ignored, especially as applied researchers tend to choose functional
 forms on the basis of parsimony and tractability.  Significance
 testing in a nonparametric framework has therefore obvious appeal as
 it requires much less restrictive assumptions.  \cite{Fan1996a},
 \cite{FanLi1996} , \cite{Racine1997}, \cite{ChenFan99},
 \cite{Lavergne2000}, \cite{Abs01}, and \cite{Delgado2001} proposed
 tests of significance for continuous variables in nonparametric
 regression models. \cite{Delgado1993}, \cite{DetteN01},
 \cite{Lavergne2001}, \cite{NeuD03}, \cite{Racine2006} focused on
 significance of discrete variables.  \cite{SigQuant13} considered
 significance testing in nonparametric quantile regression.  For each
 test, one needs first to estimate the model without the covariates
 under test, that is under the null hypothesis. The result is then
 used to check the significance of extra covariates. Two competing
 approaches are then possible.  In the ``smoothing approach,'' one
 regresses the residuals onto the whole set of covariates
 nonparametrically, while in the ``empirical process approach'' one
 uses the empirical process of residuals marked by a function of all
 covariates.

In this work, we adopt an hybrid approach to develop a new
significance test of a subset of covariates in a nonparametric
regression.  Our new test has three specific features. First, it does
not require smoothing with respect to the covariates under test as in
the ``empirical process approach.''  This allows to mitigate the curse
of dimensionality that appears with nonparametric smoothing, hence
improving the power properties of the test. Our simulation results
show that indeed our test is more powerful than competitors under a
wide spectrum of alternatives.  Second, the test statistic is
asymptotically pivotal as in the ``smoothing approach,'' while wild
bootstrap can be used to obtain small samples critical values of the
test. This yields a test whose level is well controlled by
bootstrapping, as shown in simulations.  Third, our test equally
applies whether the covariates under test are continuous or discrete,
showing that there is no need of a specific tailored procedure for
each situation.

The paper is organized as follows. In Section 2, we present our
testing procedure. In Section 3, we study its asymptotic properties
under a sequence of local alternatives and we establish the validity
of wild bootstrap.  In Section 4, we compare the small sample behavior
of our test to some existing procedures.  Section 5 gathers our proofs.

\section{Testing Framework and Procedure}

\subsection{Testing Principle}

We want to assess the significance of $X\in\mathbb{R}^{q}$ in the
nonparametric regression of $Y\in\mathbb{R}$ on $W\in\mathbb{R}^{p}$
and $X$.  Formally, this corresponds to the null hypothesis
\[
H_{0}\,:\,\mathbb{E}\left[Y\mid W,X\right]=\mathbb{E}\left[Y\mid
  W\right]\quad\mbox{a.s.}
\]
which is  equivalent to
\begin{equation}
H_{0}\,:\, \mathbb{E}\left[u\mid W,X\right]=0\quad\mbox{a.s.}
\label{NullHyp}
\end{equation}
where $u=Y-\mathbb{E}\left[Y\mid W\right]$.
The corresponding alternative hypothesis is
\[
H_{1}\,:\,\Pr\left\{ \mathbb{E}\left[u\mid W,X\right]=0\right\} <1.
\]
The following result is the cornerstone of our approach. It
characterizes the null hypothesis $H_{0}$ using a suitable
unconditional moment equation.
\begin{lem}
\label{Fundamental-Lemma}
Let $\left(W_{1},\, X_{1},\, u_{1}\right)$
and $\left(W_{2},\, X_{2},\, u_{2}\right)$ be two independent draws of
$\left(W,\, X,\, u\right)$, $\nu(W)$ a strictly positive function on
the support of $W$ such that  $\mathbb{E}[u^2 \nu^2(W)]<\infty$, and
$K(\cdot)$ and $\psi(\cdot)$  even functions with (almost
everywhere) positive Fourier integrable transforms. Define
\[
I\left(h\right)=\mathbb{E}\left[u_{1}u_{2}\nu\left(W_{1}\right)
 \nu\left(W_{2}\right)h^{-p}K\left(\left(W_{1}-W_{2}\right)/h\right)
\psi\left(X_{1}-X_{2}\right)\right]
\, .
\]
Then for any $h>0$,
\[
\mathbb{E}\left[u\mid W,X\right]= 0\,\; a.s.
\Leftrightarrow
I (h) = 0.
\]
\end{lem}
\begin{proof}
Let $\langle \cdot, \cdot \rangle$ denote the standard inner product.
Using Fourier Inversion Theorem, change of variables, and elementary
properties of conditional expectation,
\begin{eqnarray*}
\lefteqn{I\left(h\right)}
\\
 & = &
\mathbb{E}\left[u_{1}u_{2}\nu\left(W_{1}\right)\nu\left(W_{2}\right)\int_{\mathbb{R}^{p}}e^{2\pi
    i\langle t, \; W_{1}-W_{2}\rangle
  }\mathcal{F}\left[K\right]\left(th\right)dt
\right.
\\
& & \times
\left.
  \int_{\mathbb{R}^{q}}e^{2\pi i \langle s
    ,\;X_{1}-X_{2}\rangle}\mathcal{F}\left[\psi\right]\left(s\right)ds\right]\\ &
= &
\int_{\mathbb{R}^{q}}\int_{\mathbb{R}^{p}}\left|\mathbb{E}\left[\mathbb{E}\left[u\mid
    W,X\right]\nu\left(W\right)e^{2\pi i\left\{ \langle t, W
    \rangle + \langle s, X\rangle \right\}
  }\right]\right|^{2}\mathcal{F}\left[K\right]\left(th\right)\mathcal{F}\left[\psi\right]\left(s\right)dtds
\, .
\end{eqnarray*}
Since the Fourier transforms $\mathcal{F}\left[K\right]$ and
$\mathcal{F}\left[\psi\right]$  are  strictly positive, $I(h)=0$
iff
\[
\mathbb{E}\left[\mathbb{E}\left[u\mid
    W,X\right]\nu\left(W\right)e^{2\pi i\left\{ \langle t, W
    \rangle + \langle s, X\rangle \right\}
  }\right] = 0
\qquad \forall t, s
\, .
\]
But this is equivalent to $\mathbb{E}\left[u\mid
  W,X\right]\nu\left(W\right) = 0$ a.s., which by our assumption on
$\nu(\cdot)$ is equivalent to $H_{0}$.
\end{proof}

\subsection{The Test}

Lemma \ref{Fundamental-Lemma} holds whether the covariates $W$ and $X$
are continuous or discrete.  For now, we assume $W$ is continuously
distributed, and we later comment on how to modify our procedure in
the case where some of its components are discrete. We however do not
restrict $X$ to be continuous.  Since it is sufficient to test whether
$I (h)=0$ for any arbitrary $h$, we can choose $h$ to obtain
desirable properties. So we consider a sequence of $h$ decreasing to
zero when the sample size increases, which is one of the ingredient
that allows to obtain a tractable asymptotic distribution for the test
statistic.

Assume we have at hand a random sample $(Y_i,W_i,X_i)$, $1\leq i\leq
n$, from $(Y,W,X)$. In what follows, $f(\cdot)$ denotes the density of
$W$, $r(\cdot)=\mathbb{E}\left[Y\mid W=\cdot\right]$, $u=Y-r(W)$,
and $f_{i}$, $r_{i}$, $u_{i}$ respectively denote
$f\left(W_{i}\right)$, $r\left(W_{i}\right)$, and
$Y_{i}-r\left(W_{i}\right)$.  Since nonparametric estimation should be
entertained to approximate $u_{i}$, we consider usual kernel
estimators based on kernel $L(\cdot)$ and bandwidth $g$. With $L_{nik}
= \frac{1}{g^{p}}L\left(\frac{W_{i}-W_{k}}{g} \right)$, let
\begin{eqnarray*}
\hat{f}_{i} & = & (n-1)^{-1}\sum_{k\neq i, k=1}^n L_{nik}
\, ,
\\
\hat{r}_{i} & = &
\frac{1}{\hat{f}_{i}} \frac{1}{(n-1)} \sum_{k\neq i, k=1}^n Y_k L_{nik}
\,
\\
\mbox{so that } \quad \hat{u}_{i} & = & Y_i-\hat{r}_{i}=
\frac{1}{\hat{f}_{i}} \frac{1}{(n-1)} \sum_{k\neq i, k=1}^n (Y_i -Y_k) L_{nik}
\, .
\end{eqnarray*}
Denote by $n^{\left(m\right)}$ the number of arrangements of $m$
distinct elements among $n$, and by $[1/n^{\left(m\right)}]\sum_{a}$,
the average over these arrangements.  In order to avoid random
denominators, we choose $\nu\left(W\right)=f\left(W\right)$, which
fulfills the assumption of Lemma \ref{Fundamental-Lemma}.  Then we can
estimate $I\left(h\right)$ by the  second-order U-statistic
\begin{eqnarray*}
\widehat{I}_{n} & = &
\frac{1}{n^{\left(2\right)}}\sum_{a}\hat{u}_{i}\hat{f}_{i}\hat{u}_{j}\hat{f}_{j}K_{nij}\psi_{ij}\\ &
= &
\frac{1}{n^{\left(2\right)}\left(n-1\right)^{2}}\sum_{a}\sum_{k\neq
  i}\sum_{l\neq
  j}\left(Y_{i}-Y_{k}\right)\left(Y_{j}-Y_{l}\right)L_{nik}L_{njl}K_{nij}\psi_{ij}
\, ,
\end{eqnarray*}
with $K_{nij}=\frac{1}{h^{p}} K\left(\frac{W_{i}-W_{j}}{h}\right)$
and $ \psi_{ij}=\psi\left(X_{i}-X_{j}\right)$.   We also consider
the alternative statistic
\[
\tilde{I}_{n}=\frac{1}{n^{\left(4\right)}}\sum_{a}
\left(Y_{i}-Y_{k}\right)\left(Y_{j}-Y_{l}\right)L_{nik}L_{njl}K_{nij}\psi_{ij}
\, .
\]
It is clear that $\tilde{I}_{n}$ is obtained from $\widehat{I}_{n}$ by
removing asymptotically negligible ``diagonal'' terms.  Under the null
hypothesis, both statistics will have the same asymptotic normal
distribution, but removing diagonal terms reduces the bias of the
statistic under $H_{0}$.
Our statistics $\tilde{I}_{n}$ and $\widehat{I}_{n}$ are respectively
similar to the ones of \cite{FanLi1996} and \cite{Lavergne2000}, with
the fundamental difference that there is no smoothing relative to the
covariates $X$. Indeed these authors used a multidimensional smoothing
kernel over $(W,X)$, that is $h^{-\left(p+q\right)}
\tilde{K}\left(\left(W_{i}-W_{j}\right)/h, \,
\left(X_{i}-X_{j}\right)/h\right)$, while we use $K_{nij}\psi_{ij}$.
For $I_{n}$ being either $\tilde{I}_{n}$ or $\widehat{I}_{n}$, we will
show that $nh^{p/2}I_{n} \cvd \mathcal{N}\left(0,\omega^{2}\right)$
under $H_{0}$ and $nh^{p/2}I_{n}\cvp\infty$ under $H_{1}$. By
contrast, the statistics of \cite{FanLi1996} and \cite{Lavergne2000}
exhibit a $nh^{(p+q)/2}$ rate of convergence.  The alternative test of
\cite{Delgado2001} uses the kernel residuals $\hat{u}_{i}$ and the
empirical process approach of \cite{Stute1997}. This avoids extra
smoothing, but a the cost of a test statistic with a non pivotal
asymptotic law under $H_{0}$.  Hence, our proposal is an hybrid
approach that combines the advantages of existing procedures, namely
smoothing only for the variables $W$ appearing under the null
hypothesis but with an asymptotic normal distribution for the
statistic.
Given a consistent estimator $\omega^{2}_{n}$ of $\omega^{2}$, as
provided in the next section, we obtain an asymptotic
$\alpha$-level test of $H_{0}$ as
\[
\mbox{Reject } H_{0} \mbox{ if } \
nh^{p/2}I_{n} / \omega_{n} > z_{1-\alpha}
\, ,
\]
where $z_{1-\alpha}$ is the $(1-\alpha)$-th quantile of the standard
normal distribution. In small samples, we will show the validity of a
wild bootstrap scheme to obtain critical values.

The test applies whether $X$ is continuous or has some discrete
components.  The procedure is also easily adapted to some discrete
components of $W$. In that case, one would replace kernel smoothing by
cells' indicators for the discrete components, so that for $W$
composed of continuous $W_{c}$ of dimension $p_c$ and discrete
$W_{d}$, one would use $h^{-p_c} K\left(\frac{W_{ic}-W_{jc}}{h}\right)
\ind(W_{id}=W_{jd})$ instead of $ K_{nij}$. It would also be possible
to smooth on the discrete components, as proposed by
\cite{RacineLi2004}.  To obtain scale invariance, we recommend that
observations on covariates should be scaled, say by their sample
standard deviation as is customary in nonparametric estimation. It is
equally  important to scale the $X_{i}$ before they are used as
arguments of $\psi(\cdot)$ to preserve such invariance.

The outcome of the test may depend on the choice of the kernels
$K(\cdot)$ and $L(\cdot)$, while this influence is expected to be
limited as it is in nonparametric estimation.  The choice of the
function $\psi(\cdot)$ might be more important, but our simulations
reveal that it is not.  From our theoretical study, this function, as
well as $K(\cdot)$ should possess an almost everywhere positive and
integrable Fourier transform. This is true for (products of) the
triangular, normal, Laplace, and logistic densities, see \cite{JKB}, and for a Student
density, see \cite{Hurst}. Alternatively, one can choose $\psi (x)$ as
a univariate density applied to some transformation of $x$, such as
its norm. This yields $\psi (x) = g \left( \| x\|\right)$ where
$g(\cdot)$ is any of the above univariate densities.  This is the form
we will consider in our simulations to study the influence of
$\psi(\cdot)$.

\section{Theoretical Properties}

We here give the asymptotic properties of our test statistics under
$H_{0}$ and some local alternatives. To do so in a compact way, we
consider the sequence of hypotheses
\[
H_{1n}\,:\,\mathbb{E}\left[Y\mid
  W,X\right]=r\left(W\right)+\delta_{n}d\left(W,X\right),\qquad n\geq
1,
\]
where $d(\cdot)$ is a fixed integrable function.  Since
$r\left(W\right) = \mathbb{E}\left[Y\mid W\right]$, our setup imposes
$\mathbb{E}\left[d\left(W,X\right)\mid W\right]=0$.  The null
hypothesis corresponds to the case $\delta_n\equiv 0$, while
considering a sequence $\delta_{n}\to 0$ yields local Pitman-like
alternatives.

\subsection{Assumptions}

We begin by some useful definitions.
\begin{defn}
\label{RegulDef}
\begin{description}
\item [{(i)}] $\mathcal{U}^{p}$ is the class of integrable uniformly continuous
functions from $\mathbb{R}^{p}$ to $\mathbb{R}$;
\item [{(ii)}] $\mathcal{D}_{s}^{p}$ is the class of $m$-times differentiable
functions from $\mathbb{R}^{p}$ to $\mathbb{R}$ , with derivatives
of order $\left\lfloor s\right\rfloor $ that are uniformly Lipschitz
continuous of order $s-\left\lfloor s\right\rfloor $,
where $\left\lfloor s\right\rfloor $ denotes the integer such that
$\left\lfloor s\right\rfloor \leq s < \left\lfloor s\right\rfloor +1$.
\end{description}
\end{defn}
\medskip
Note that a function belonging to $\mathcal{U}^{p}$ is necessarily bounded.
\begin{defn}
\label{KerDef}$\mathcal{K}_{m}^{p}$, $m\geq2$, is the class of
even integrable functions $K\,:\,\mathbb{R}^{p}\to\mathbb{R}$  with
compact support satisfying $\int K\left(t\right)dt=1$ and, if $t=(t_1,\dots,t_p)$,
\[
\int_{\mathbb{R}^{p}} t_{1}^{\alpha_{1}}\dots t_{p}^{\alpha_{p}}
K\left(t\right)dt=0\;\; \mbox{ for }\;\;
0<\sum_{i=1}^{p}\alpha_{i}\leq m-1,\,\alpha_{i}\in\mathbb{N}\quad\forall i
\]
\end{defn}
This definition of higher-order kernels is standard in nonparametric
estimation.  The compact support assumption is made for simplicity and
could be relaxed at the expense of technical conditions on the rate of
decrease of the kernels at infinity, see e.g. Definition 1 in
\cite{FanLi1996}. In particular, the gaussian kernel could be allowed
for.
We are now ready to list our assumptions.
\begin{assumption}
\label{Sample} (i) For any $x\in\mathbb{R}^q$ in the support of $X$, the vector $W$
 admits a conditional density given $X=x$ with respect to the Lebesgue
 measure in $\mathbb{R}^p$, denoted by $\pi(\cdot\mid x)$. Moreover,
 $\mathbb{E}\left[Y^{8}\right]<\infty$.  (ii) The observations
 $\left(W_{i},X_{i},Y_{i}\right)$, $i=1,\cdots, n $ are independent
 and identically distributed as $(W,X,Y)$.
\end{assumption}
The existence of the conditional density given $X=x$ for all
$x\in\mathbb{R}^q$ in the support of $X$ implies that $W$ admits a
density with respect to the Lebesgue measure on $\mathbb{R}^p$.  As
noted above, our results easily generalizes to some discrete
components of $W$, but for the sake of simplicity we do not formally
consider this in our theoretical analysis.
\begin{assumption}
\label{RegulHyp}
\begin{description}
\item[{(i)}] $f\left(\cdot\right)$ and
  $r\left(\cdot\right)f\left(\cdot\right)$ belong to
  $\mathcal{U}^{p}\cap\mathcal{D}_{s}^{p}$, $s\geq2$;
\item [{(ii)}] $\mathbb{E}\left[u^{2}\mid W=\cdot\right]f\left(\cdot\right)$,
$\mathbb{E}\left[u^{4}\mid W = \cdot\right]f^{4}\left(\cdot\right)$
  belong to $\mathcal{U}^{p};$
\item[{(iii)}] the function $\psi\left(\cdot\right)$ is bounded and
  has a almost everywhere positive  and integrable Fourier transform;
\item[{(iv)}]  $K\left(\cdot\right)\in\mathcal{K}_{2}^{p}$ and has an almost
  everywhere positive and integrable Fourier transform, while
  $L\left(\cdot\right)\in\mathcal{K}_{\left\lfloor s\right\rfloor}^{p}$
and  is  of bounded variation;
\item[{(v)}] let $\sigma^2(w,x)=\mathbb{E}[u^2 \mid W=w,X=x]$, then
  $\sigma^{2}\left(\cdot,x\right)
  f^{2}\left(\cdot\right)\pi\left(\cdot\mid x\right)$ belongs to
  $\mathcal{U}^{p}$ for any $x$ in the support of $X$, has integrable
  Fourier transform, and

  $\mathbb{E}\left[\sigma^{4}\left(W,X\right)f^{4}\left(W\right)\pi\left(W\mid
    X\right)\right]<\infty$;
\item[{(vi)}]  $\mathbb{E}[d^2(W,X)\mid W=\cdot]
  f^2(\cdot)$ belongs to $\mathcal{U}^{p}$,
 $d(\cdot , x)
  f\left(\cdot\right)\pi\left(\cdot\mid x\right)$ is integrable and
  squared integrable for any $x$ in the support of $X$, and

  $\mathbb{E}\left[d^{2}\left(W,X\right)f^2\left(W\right)\pi\left(W\mid
    X\right)\right]<\infty$.
\end{description}
\end{assumption}
Standard regularity conditions are assumed for various functions.  A
higher-order kernel $L(\cdot)$ is used in conjunction with the
differentiability conditions in (i) to ensure that the bias in
nonparametric estimation is small enough.

\subsection{Asymptotic Analysis}

The following result characterizes the behavior of our statistics
under the null hypothesis and a sequence of local alternatives.
\begin{thm}
\label{Consistency} Let $I_n$ be any of the statistics $\widehat{I}_n$ or $\tilde{I}_n$.
Under Assumptions \ref{Sample} and \ref{RegulHyp}, and if as
$n\to\infty$ (i) $g,h\to 0$, { (ii) $n^{7/8} g^{p}/\ln n ,$} $nh^{p}\to\infty$,
(iii) $nh^{p/2}g^{2s}\to 0$, and (iv) $h/g\to 0$ if $I_n =
\tilde{I}_n$ or $h/g^2\to 0$ if $I_n = \widehat{I}_n$, then
\begin{description}
\item [{(i)}] If $\delta_{n}^{2}nh^{p/2}\to C$ with $0\leq C<\infty$,
  $nh^{p/2}I_{n} \cvd \mathcal{N}\left(C\mu,\omega^{2}\right)$
  where
\[
  \mu=\mathbb{E}\left[\intop\!\!
    d\left(w,X_{1}\right)d\left(w,X_{2}\right)f^{2}\left(w\right)\pi\left(w\mid
    X_{1}\right)\pi\left(w\mid X_{2}\right)\psi\left(X_{1}\!-\!
    X_{2}\right)dw\right]>0
\]
\begin{eqnarray*}
\lefteqn{\mbox{and } \quad \omega^{2} =
2 \int\! K^{2}\!\left(s\right) \, ds}
\\
& &
\mathbb{E}\left[\int\!\!\sigma^{2}\left(w,X_{1}\right)
  \sigma^{2}\left(w,X_{2}\right)f^{4}\left(w\right)\pi\left(w\mid
  X_{1}\right)\pi\left(w\mid X_{2}\right)\psi^{2}\left(X_{1}\!-\!
  X_{2}\right)dw\right]
\, .
\end{eqnarray*}
\item [{(ii)}] If $\delta_{n}^{2}nh^{p/2}\to\infty$, $nh^{p/2}I_{n}\cvp\infty$.
\end{description}
\end{thm}
The rate of convergence of the test statistic depends only on the
dimension of $W$, the covariates present under the null hypothesis,
but not on the dimension of $X$, the covariates under test. Similarly,
the rate of local alternatives that are detected by the test depends
only on the dimension of $W$. As shown in the simulations, this yields
some gain in power compared to competing ``smoothing''
tests. Conditions (i) to (iv) together require that $s>p/2$ for
$I_{n}=\tilde{I}_{n}$ and $s>p/4$ for $I_{n}=\widehat{I}_{n}$, so
removing diagonal terms in $\widehat{I}_{n}$ allows to weaken the
restrictions on the bandwidths.  { Condition (ii) could be slightly
  weakened to $ng^p \rightarrow \infty$ at the price of handling high
  order $U$-statistics in the proofs, but allows for a shorter
  argument based on empirical processes, see Lemma
  \ref{unif_omeg} in the proofs section.}

\bigskip
To estimate $\omega^{2}$, we can either mimic \cite{Lavergne2000} to
consider
\[
\tilde{\omega}_{n}^{2}=\dfrac{2h^{p}}{n^{\left(6\right)}}\sum_{a}
\left(Y_{i}-Y_{k}\right)\left(Y_{i}-Y_{k^{\prime}}\right)\left(Y_{j}-Y_{l}\right)
\left(Y_{j}-Y_{l^{\prime}}\right)L_{nik}L_{nik^{\prime}}L_{njl}L_{njl^{\prime}}K_{nij}^{2}\psi_{ij}^{2},
\]
or generalize the variance estimator of \cite{FanLi1996} as
\[
\widehat{\omega}_{n}^{2}=\dfrac{2h^{p}}{n^{\left(2\right)}}\sum_{a}
\hat{u}_{i}^{2}\hat{f}_{i}^{2}\hat{u}_{j}^{2}\hat{f}_{j}^{2}K_{nij}^{2}\psi_{ij}^{2}.
\]
The first one is  consistent for $\omega^{2}$ under both the
null and alternative hypothesis, but the latter is faster to compute.
\begin{cor}\label{corr_test_o}
Let $I_n$ be any of the statistics $\widehat{I}_n$ or $\tilde{I}_n$ and let
$\omega_{n}$ denote any of $\widehat{\omega}_{n}$ or $\tilde{\omega}_{n}$.
Under the assumptions of Theorem \ref{Consistency}, the test that rejects
$H_{0}$ when $nh^{p/2}I_{n}/\omega_{n} > z_{1-\alpha}$ is of asymptotic
level $\alpha$ under $H_{0}$ and is consistent under the sequence of local
alternatives $H_{1n}$ provided $\delta_{n}^{2}nh^{p/2}\to\infty$.
\end{cor}

\subsection{Bootstrap Critical Values}

It is known that asymptotic theory may be inaccurate for small and
moderate samples when using smoothing methods. Hence, as in e.g.
\cite{Hardle1993} or \cite{Delgado2001}, we consider a wild bootstrap
procedure to approximate the quantiles of our test statistic.
Resamples are obtained from $Y_{i}^{*}=\hat{r}_{i}+u_{i}^{*}$, where
$u_{i}^{*}=\eta_{i}\hat{u}_{i}$ and $\eta_{i}$ are i.i.d.  variables
independent of the initial sample with $\mathbb{E}\eta_{i}=0$ and
$\mathbb{E}\eta_{i}^{2}=\mathbb{E}\eta_{i}^{3}=1$, $1\leq i\leq n$.
The $\eta_i$ could for instance follow the two-point law of
\cite{Mammen1993}. With at hand a bootstrap sample
$(Y_{i}^{*},W_i,X_i)$, $1\leq i\leq n$, we obtain a bootstrapped
statistic $I_{n}^{*}$ with bootstrapped observations $Y_{i}^{*}$ in
place of original observations $Y_{i}$.  When the scheme is repeated
many times, the bootstrap critical value $z^\star_{1-\alpha, n}$ at
level $\alpha$ is the empirical $(1-\alpha)$-th quantile of the
bootstrapped test statistics. The asymptotic validity of this
bootstrap procedure is guaranteed by the following result.
\begin{thm}
\label{Bootstrap Consistency}
Suppose Assumptions \ref{Sample}, \ref{RegulHyp}, and Conditions (i)
to (iii) of Theorem \ref{Consistency} hold. Moreover, assume
$\inf_{w\in\mathcal{S}_{W}}f\left(w\right)>0$ and $h/g^{2}\to
0$. Then
for $I_{n}^{*}$ equal to any of $\widehat{I}_{n}^{*}$ and $\tilde{I}_{n}^{*}$,
\[
\sup_{z\in\mathbb{R}}\left|\mathbb{P}\left[nh^{p/2}I_{n}^{*}/\omega_{n}^{*}\leq
  z\mid Y_1,W_{1}, X_{1},\cdots, Y_n, W_{n},
  X_{n}\right]-\Phi\left(z\right)\right| \cvp 0 \, ,
\]
where $\Phi\left(\cdot\right)$ is the standard normal distribution function.
\end{thm}

\section{Monte Carlo Study}

We investigated the small sample behavior of our test
and studied its performances relative to alternative tests.
We generated data through
\[
Y=\left(W'\theta\right)^{3} - W'\theta+ \delta d\left(X\right) + \varepsilon
\]
where $W$ follow a two-dimensional standard normal, $X$ independently
follows a $q$-variate standard normal, $\varepsilon\sim
\mathcal{N}\left(0, 4\right)$, and we set
$\theta=\left(1,\,-1\right)^{\prime}/\sqrt{2}$.  The null hypothesis
corresponds to $\delta = 0$, and we considered various forms for
$d(\cdot)$ to investigate power.  We only considered the test based on
$\tilde{I}_n$, labelled LMP, as preliminary simulation results showed
that it had similar or better performances than the test based on
$\widehat{I}_n$. We compared it to the test of Lavergne and Vuong
(2000, hereafter LV), and the test of Delgado and Gonzalez-Manteiga
(2001, hereafter DGM). The statistic for the latter test
is  the Cramer-von-Mises statistic
\[
\sum_{i=1}^{n} \left[ \sum_{j=1}^{n}{ \widehat{u}_{j}} \widehat{f}_{j}
  \,\mathbf{1}\left\{W_j \leq W_{i}\right\} \,\mathbf{1}\left\{X_j
  \leq X_{i}\right\} \right]^{2}
\, ,
\]
and critical values are obtained by wild bootstrapping as for our own
statistic.  To compute bootstrap critical values, we used 199
bootstrap replications and the two-point distribution
\[
  \Pr \left( \eta_{i} = \frac{1-\sqrt{5}}{2}  \right)
  =
  \frac{5+\sqrt{5}}{10}
  \; , \;
  \Pr \left( \eta_{i} = \frac{1+\sqrt{5}}{2} \right)
  =
  \frac{5-\sqrt{5}}{10}
  \; .
\]
For all tests, each time a kernel appears, we used the Epanechnikov
kernel applied to the norm of its argument $u$, that is
$0.75\,\left(1-\left\Vert u\right\Vert^2\right)\mathbf{1}\left\{
\left\Vert u\right\Vert <1\right\}$.  The bandwidth parameters are set
to $g=n^{-1/6}$ and $h=c \, n^{-2.1/6}$, and we let $c$ vary to
investigate the sensitivity of our results to the smoothing
parameter's choice.  To study the influence of $\psi(\cdot)$ on our
test, we considered $\psi(x) = l\left(\left\Vert x\right\Vert
\right)$, where $l(\cdot)$ is a triangular or normal density, each
with a second moment equal to one.

Figure \ref{fig:LevelCont} reports the empirical level of the various
tests for $n=100$ based on 5000 replications when we let $c$ and $q$
vary. For our test, bootstrapping yields more accurate rejection levels
than the asymptotic normal critical values for any bandwidth factor
$c$ and dimension $q$.  The choice of $\psi(\cdot)$ does not influence
the results.  The empirical level of LV test is much more sensitive to
the bandwidth and the dimension. The empirical level of the DGM test
is close to the nominal one for a low dimension $q$, but decreases
with increasing $q$.

To investigate power, we considered different forms of alternatives as
specified by $d(\cdot)$. We first focus on a quadratic alternative,
where $d\left(X\right)= \left(X'\beta-1\right)^{2}/\sqrt{2}$, with
$\beta=\left(1,\,,1,\,,\dots\right)^{\prime}/\sqrt{q}$.  Figure
\ref{fig:PowerContQs} reports power curves of the different tests for
the quadratic alternative, $n=100$, and a nominal level of 10\% based
on $2000$ replications.  We also report the power of a Fisher test
based on a linear specification in the components of $X$.  The power
of our test, as well as the one of LV test, increases when the
bandwidth factor $c$ increases.  This is in line with theoretical
findings, though we may expect this relationship to revert for very
large bandwidths.  Our test always dominates LV test, as well as the
Fisher test and DGM test, for any choice of $c$ and any dimension $q$.
The power of all tests decreases when the dimension $q$ increases, but
the more notable degradation is for the DGM test.  In Figure
\ref{fig:PowerContNs}, we let $n$ vary for a fixed dimension $q=5$.
The power of all tests improve, but our main qualitative findings are
not affected.  It is noteworthy that the power advantage of our test
compared to LV test become more pronounced as $n$ increases.  In
Figure \ref{fig:PowerContAlter}, we considered a linear alternative
$d\left(X\right)=X'\beta$ and a sine alternative, $d\left(X\right)=
\sin\left(2 \, X'\beta\right)$. Our main findings remain
unchanged. For a linear alternative, the Fisher test is most powerful
as expected. Compared to this benchmark, the loss of power when using
our test is moderate for a large enough bandwidth factors $c$. For a
sine alternative, our test is more powerful than the Fisher test for
$c=2$ or 4.

We also considered the case of a discrete $X$.  We
generated data following
\[
Y=\left(W'\theta\right)^{3} - W'\theta+ \delta d\left(W\right) \,\mathbf{1}\left\{X=1\right\} + \varepsilon
\]
where $W$ and $\varepsilon$ are generated as before, and $X$ is
Bernoulli with probability of success $p=0.6$.  We compared our test to
two competitors.  The test proposed by Lavergne (2001) is
similar to our test with the main difference that $\psi(\cdot)$ is the
indicator function, i.e. $\psi\left( X_i - X_j\right) =
\mathbf{1}\left\{X_i=X_j\right\}$.
The test of Neumeyer et Dette (2003, hereafter ND) is similar in
spirit to the DGM test.  The details of the simulations are similar to
above.  Figures \ref{fig:LevelDisc} and \ref{fig:PowerDisc} report our
results. Bootstrapping our test and Lavergne's test yield accurate
rejection levels, while the asymptotic tests and the ND test
underrejects. Under a quadratic alternative, the power of our test is
comparable to the one of the ND test for a large enough bandwidth
factor $c$. Under a sine alternative, our test outperforms ND test in
all cases.

\section{Conclusion}

We have developed a testing procedure for the significance of
covariates in a nonparametric regression. Smoothing is entertained
only for the covariates under the null hypothesis. The resulting test
statistic is asymptotically pivotal, and wild bootstrap can be used to
obtain critical values in small and moderate samples.  The test is
versatile, as it applies whether the covariates under test are
continuous and/or discrete.  Simulations reveal that our test
outperforms its competitors in many situations, and especially when
the dimension of covariates is large.



\section{Proofs}

We here provide the proofs of the main results. Technical lemmas are
relegated to the Appendix.

In the following, for any integrable function $\delta(X),$ let
$\mathcal{F}_{X}\left[\delta\right]\left(u\right)=\mathbb{E}[e^{ -2\pi
    i\langle X, \;u\rangle} \delta\left(X\right)],$
$u\in\mathbb{R}^q.$ Moreover, for any index set $I$ not containing $i$
with cardinality $\left|I\right|$,
define $$\widehat{f_{i}^{I}}=\left(n-\left|I\right|-1\right)^{-1}\sum_{k\neq
  i,k\notin I}L_{nik},$$ consistent with $\widehat{f_{i}}$ that
corresponds to the case where $I$ is the empty set.


\subsection{Proof of Theorem \ref{Consistency}}

 We first consider the case $I_n=\tilde{I}_{n}$. Next, we study the
 difference between $\tilde{I}_{n}$ and $\widehat{I}_{n}$ and hence
 deduce the result for $I_n=\widehat{I}_{n}$.

\paragraph{Case  $I_n=\tilde{I}_{n}$.}
Consider the decomposition
\begin{eqnarray*}
I_{n} & = &
\frac{1}{n^{\left(4\right)}}\sum_{a}\left(u_{i}-u_{k}\right)\left(u_{j}-u_{l}\right)L_{nik}L_{njl}K_{nij}\psi_{ij}\\ &
&
+\frac{2}{n^{\left(4\right)}}\sum_{a}\left(u_{i}-u_{k}\right)\left(r_{j}-r_{l}\right)L_{nik}L_{njl}K_{nij}\psi_{ij}\\ &
&
+\frac{1}{n^{\left(4\right)}}\sum_{a}\left(r_{i}-r_{k}\right)\left(r_{j}-r_{l}\right)L_{nik}L_{njl}K_{nij}\psi_{ij}\\ &
= & I_{1}+2I_{2}+I_{3},
\end{eqnarray*}
where
\begin{eqnarray*}
I_{1} & = & \frac{n-2}{n-3}\frac{1}{n^{(2)}}\sum_{a}u_{i}u_{j}f_{i}f_{j}K_{nij}\psi_{ij}+\frac{2\left(n-2\right)}{n-3}\frac{1}{n^{\left(2\right)}}\sum_{a}u_{i}\left(\widehat{f}_{i}^{j}-f_{i}\right)u_{j}f_{j}K_{nij}\psi_{ij}\\
 &  & +\frac{n-2}{n-3}\frac{1}{n^{(2)}}\sum_{a}u_{i}\left(\widehat{f}_{i}^{j}-f_{i}\right)u_{j}\left(\widehat{f}_{j}^{i}-f_{j}\right)K_{nij}\psi_{ij}-\frac{2}{n^{(3)}}\sum_{a}u_{i}f_{i}u_{l}L_{njl}K_{nij}\psi_{ij}\\
 &  & -\frac{2}{n^{\left(3\right)}}\sum_{a}u_{i}\left(\widehat{f}_{i}^{j,l}-f_{i}\right)u_{l}L_{njl}K_{nij}\psi_{ij}+\frac{1}{n^{\left(4\right)}}\sum_{a}u_{k}u_{l}L_{nik}L_{njl}K_{nij}\psi_{ij}\\
 &  & -\frac{1}{n^{\left(4\right)}}\sum_{a}u_{i}u_{j}L_{nik}L_{njk}K_{nij}\psi_{ij}\\
 & = & \frac{n-2}{n-3}\left[I_{0n}+2I_{1,1}+I_{1,2}\right]-2I_{1,3}-2I_{1,4}+I_{1,5}-I_{1,6},
\end{eqnarray*}
and
\begin{eqnarray*}
I_{2} & = & \frac{1}{n^{\left(3\right)}}\sum_{a}u_{i}f_{i}\left(r_{j}-r_{l}\right)L_{njl}K_{nij}\psi_{ij}+\frac{1}{n^{\left(3\right)}}\sum_{a}u_{i}\left(\widehat{f}_{i}^{j,l}-f_{i}\right)\left(r_{j}-r_{l}\right)L_{njl}K_{nij}\psi_{ij}\\
 &  & -\frac{1}{n^{\left(4\right)}}\sum_{a}u_{k}\left(r_{j}-r_{l}\right)L_{nik}L_{njl}K_{nij}\psi_{ij}=I_{2,1}+I_{2,2}-I_{2,3}.
\end{eqnarray*}

In Proposition \ref{Normality} we prove that, under $H_{0},$ $I_{0n}$
is asymptotically centered Gaussian with variance $\omega^{2}$, while
in Proposition \ref{NormalityH1} we prove that, under $H_{1n},$
$I_{0n}$ is asymptotically Gaussian with mean $\mu$ and variance
$\omega^{2}$ provided $\delta_{n}^{2}nh^{p/2}$ converges to some
positive real number.  In Propositions \ref{Ustat} and \ref{Remaining}
we show that all remaining terms in the decomposition of $I_n$ are
asymptotically negligible.

\begin{prop}
\label{Normality}Under the conditions of Theorem \ref{Consistency},
$nh^{p/2}I_{0n} \cvd \mathcal{N}\left(0,\omega^{2}\right)$ under
$H_{0}$.
\end{prop}
\begin{proof}
Let us define the martingale array $\left\{
S_{n,m},\mathcal{F}_{n,m},\,1\leq m\leq n,\, n\geq1\right\} $ where
$S_{n,1}=0,$ and
\[
S_{n,m}=\sum_{i=1}^{m}G_{n,i} \;\;\; \text{ with } \; \;\; G_{n,i}=\dfrac{2h^{p/2}}{n-1}u_{i}f_{i}\sum_{j=1}^{i-1}u_{j}f_{j}K_{nij}\psi_{ij},
\qquad 2\leq i, m \leq n,
\]
and $\mathcal{F}_{n,m}$ is the $\sigma-$field generated by $\left\{ W_{1},\,\dots,\, W_{n},\, X_{1},\,\dots,\, X_{n},\, Y_{1},\,\dots,\, Y_{m}\right\} .$ Thus $nh^{p/2}I_{0n}=S_{n,n}$.
Also define
$$
V_{n}^{2}  =  \sum_{i=2}^{n}E\left[G_{n,i}^{2}\mid\mathcal{F}_{n,i-1}\right]
 =  \dfrac{4h^{p}}{\left(n-1\right)^{2}}\sum_{i=2}^{n}\sigma_{i}^{2}f_{i}^{2}\left(\sum_{j=1}^{i-1}u_{j}f_{j}K_{nij}\psi_{ij}\right)^{2}
$$
where $\sigma_{i}^{2}=\sigma^{2}\left(W_{i},X_{i}\right)$. We can
decompose $V_{n}^{2}$ as
\begin{eqnarray*}
V_{n}^{2} & = & \dfrac{4h^{p}}{\left(n-1\right)^{2}}\sum_{i=2}^{n}\sigma_{i}^{2}f_{i}^{2}\sum_{j=1}^{i-1}\sum_{k=1}^{i-1}u_{j}f_{j}u_{k}f_{k}K_{nij}K_{nik}\psi_{ij}\psi_{ik}\\
 & = & \dfrac{4h^{p}}{\left(n-1\right)^{2}}\sum_{i=2}^{n}\sum_{j=1}^{i-1}\sigma_{i}^{2}f_{i}^{2}u_{j}^{2}f_{j}^{2}K_{nij}^{2}\psi_{ij}^{2}\\
 &  & +\dfrac{8h^{p}}{\left(n-1\right)^{2}}\sum_{i=3}^{n}\sum_{j=2}^{i-1}\sum_{k=1}^{j-1}\sigma_{i}^{2}f_{i}^{2}u_{j}f_{j}u_{k}f_{k}K_{nij}K_{nik}\psi_{ij}\psi_{ik}
 =  A_{n}+B_{n}.
\end{eqnarray*}
The result follows from the Central Limit Theorem for martingale
arrays, see Corollary 3.1 of \cite{Hall1980}. The conditions required
for Corollary 3.1 of \cite{Hall1980}, among which
$V_{n}^{2}\cvp\omega^{2}$, are checked in Lemma \ref{AAn} below.  Its
proof is provided in the Appendix.
\begin{lem}
\label{AAn} Under the conditions  of Proposition \ref{Normality},
\begin{enumerate}
\item\label{An} $A_{n}\cvp \omega^{2}$,
\item\label{Bn}   $B_{n}\cvp 0$,
\item\label{Lindebergh} the martingale difference array
$\left\{ G_{n,i},\,\mathcal{F}_{n,i},\,1\leq i\leq n\right\} $
satisfies the Lindeberg condition
\begin{equation*}
\forall\varepsilon>0,\quad\sum_{i=2}^{n}\mathbb{E}
\left[G_{n,i}^{2}I\left(\left|G_{n,i}\right|>\varepsilon\right)\mid\mathcal{F}_{n,i-1}\right]
\cvp 0 \, .
\end{equation*}
\end{enumerate}
\end{lem}
\end{proof}

\begin{prop}\label{NormalityH1}
Under the conditions of Theorem \ref{Consistency} and $H_{1n},$ if
$\delta_{n}^{2}nh^{p/2}\rightarrow C$ with $0<C<\infty,$
$nh^{p/2}I_{0n} \cvd \mathcal{N}\left(C \mu,\omega^{2}\right)$.
\end{prop}
\begin{proof}
Let $\varepsilon_{i}=Y_{i}-\mathbb{E}\left[Y_{i}\mid W_{i},\,
  X_{i}\right]$ and let us decompose
\begin{eqnarray*}
nh^{p/2}I_{0n} & = & \dfrac{h^{p/2}}{n-1}\sum_{i=1}^{n}\sum_{j\neq i}u_{i}f_{i}u_{j}f_{j}K_{nij}\psi_{ij}\\
 & = & \dfrac{h^{p/2}}{n-1}\sum_{i=1}^{n}\sum_{j\neq i}\left(\delta_{n}d_{i}+\varepsilon_{i}\right)f_{i}\left(\delta_{n}d_{j}+\varepsilon_{j}\right)f_{j}K_{nij}\psi_{ij}\\
 & = & \dfrac{h^{p/2}}{n-1}\sum_{i=1}^{n}\sum_{j\neq i}\varepsilon_{i}f_{i}\varepsilon_{j}f_{j}K_{nij}\psi_{ij}+\dfrac{\delta_{n}h^{p/2}}{n-1}\sum_{i=1}^{n}\sum_{j\neq i}d_{i}f_{i}\left(\delta_{n}d_{j}+2\varepsilon_{j}\right)f_{j}K_{nij}\psi_{ij}\\
 &=& C_{0n}+C_{n}.
\end{eqnarray*}
By Proposition \ref{Normality}, $C_{0n}\cvd
\mathcal{N}\left(0,\omega^{2}\right). $ As for $C_{n}$, we have
$$
\mathbb{E}\left[C_{n}\right] =
\delta_{n}^{2}nh^{p/2}\mathbb{E}\left[d_{i}f_{i}d_{j}f_{j}K_{nij}\psi_{ij}\right]=
\delta_{n}^{2}nh^{p/2}\mu_{n}
\, .
$$
By repeated application of Fubini's Theorem, Fourier Inverse formula,
Dominated Convergence Theorem, and Parseval's identity, we obtain
\begin{eqnarray*}
\mu_{n} & = &
\mathbb{E}\left[d_{1}f_{2}d_{1}f_{2}K_{n12}\psi_{12}\right]\\ & = &
\mathbb{E}\left[\iint
  d\left(w_{1},X_{1}\right)d\left(w_{2},X_{2}\right)f\left(w_{1}\right)f\left(w_{2}\right)f\left(w_{1}|X_{1}\right)f\left(w_{2}|X_{2}\right)\right.\\ &&\qquad
  \qquad \qquad \qquad\qquad \qquad\qquad \qquad\qquad \qquad\times
  \left. h^{-p}K\left(\dfrac{w_{1}-w_{2}}{h}\right)dw_{1}dw_{2}\;\,\psi\left(X_{1}-X_{2}\right)\right]\\ &
= & \mathbb{E}\left[
  \int\!\!\mathcal{F}\!\left[d\left(\cdot,X_{1}\right)\!f\left(\cdot\right)\pi\left(\cdot\mid
    X_{1}\right)\right]\!\left(t\right)\mathcal{F}\left[d\left(\cdot,X_{2}\right)\!f\left(\cdot\right)\pi\left(\cdot\mid
    X_{2}\right)\right]\!\left(-t\right)\mathcal{F}\left[K\right]\left(ht\right)dt\;\psi\left(X_{1}-X_{2}\right)\right]\\ &
\to & \mathbb{E}\left[ \left[
    \int\mathcal{F}\left[d\left(\cdot,X_{1}\right)f\left(\cdot\right)\pi\left(\cdot\mid
      X_{1}\right)\right]\left(t\right)\mathcal{F}\left[d\left(\cdot,X_{2}\right)f\left(\cdot\right)\pi\left(\cdot\mid
      X_{2}\right)\right]\left(-t\right)dt\;
    \right]\psi\left(X_{1}-X_{2}\right)\right]\\ & = &
\mathbb{E}\left[\int
  d\left(w,X_{1}\right)d\left(w,X_{2}\right)f^{2}\left(w\right)\pi\left(w\mid
  X_{1}\right)\pi\left(w\mid
  X_{2}\right)\psi\left(X_{1}-X_{2}\right)dw\right]\\ &=&\int
\left[\int \mathcal{F}_X\left[
    d\left(w,\cdot\right)\pi\left(w\mid\cdot\right)\right](u)
  \mathcal{F}_X\left[d\left(w,\cdot\right)\pi\left(w\mid\cdot\right)\right](-u)\mathcal{F}[\psi](u)du
  \right] f^{2}\left(w\right)dw\\ &=& \iint \left| \mathcal{F}_X\left[
  d\left(w,\cdot\right)\pi\left(w\mid\cdot\right)\right](u)\right|^2
\mathcal{F}[\psi](u) f^{2}\left(w\right)du dw =\mu \, .
\end{eqnarray*}
Moreover,
\begin{eqnarray*}
\mbox{Var}\left[C_{n}\right] & \leq & \dfrac{4\delta_{n}^{4}h^{p}}{\left(n-1\right)^{2}}\sum_{a}\mathbb{E}\left[d_{i}^{2}f_{i}^{2}d_{k}d_{l}f_{k}f_{l}K_{nik}K_{nil}\psi_{ik}\psi_{il}\right]\\
 &  & +\dfrac{2\delta_{n}^{4}h^{p}}{\left(n-1\right)^{2}}\sum_{a}\mathbb{E}\left[d_{i}^{2}f_{i}^{2}d_{k}^{2}f_{k}^{2}K_{nik}^{2}\psi_{ik}^{2}\right]\\
 &  & +\dfrac{4\delta_{n}^{2}h^{p}}{\left(n-1\right)^{2}}\sum_{a}\mathbb{E}\left[d_{i}f_{i}d_{j}f_{j}\varepsilon_{k}^{2}f_{k}^{2}K_{nik}K_{njk}\psi_{ik}\psi_{jk}\right]\\
 &  & +\dfrac{4\delta_{n}^{2}h^{p}}{\left(n-1\right)^{2}}\sum_{a}\mathbb{E}\left[d_{i}^{2}f_{i}^{2}\varepsilon_{k}^{2}f_{k}^{2}K_{nik}^{2}\psi_{ik}^{2}\right]\\
 & = & O\left(\delta_{n}^{4}nh^{p}\right)+O\left(\delta_{n}^{4}\right)+O\left(\delta_{n}^{2}nh^{p}\right)+O\left(\delta_{n}^{2}\right).
\end{eqnarray*}
Therefore $C_{n}=C\mu_{n}+O_p\left(\delta_{n} n^{1/2} h^{p/2}\right)
\cvp C \mu$, and the desired result follows.
\end{proof}

\bigskip
\begin{prop}
\label{Ustat} Under the conditions of Theorem \ref{Consistency},
\begin{description}
\item [{(i)}] $nh^{p/2}I_{1,3}=\delta_{n}\sqrt{n}h^{p/2}O_{p}\left(1\right)+o_{p}\left(1\right)$,
\item [{(ii)}] $nh^{p/2}I_{1,5}=o_{p}\left(1\right)$,
\item [{(iii)}] $nh^{p/2}I_{1,6}=\delta_{n}^{2}nh^{p/2}o_{p}\left(1\right)+o_{p}\left(1\right)$,
\item [{(iv)}] $nh^{p/2}I_{2,1}=\delta_{n}\sqrt{n}h^{p/2}o_{p}\left(1\right)+\delta_{n}\sqrt{n}h^{p/2}g^{s}O_{p}\left(1\right)+o_{p}\left(1\right)$,
\item [{(v)}] $nh^{p/2}I_{2,3}=o_{p}\left(1\right)$,
\item [{(vi)}] $nh^{p/2}I_{3}=nh^{p/2}O_{p}\left(g^{2s}\right)+o_{p}\left(1\right)$.
\end{description}
\end{prop}
\medskip
\begin{prop}
\label{Remaining} Under the conditions of Theorem \ref{Consistency},
\begin{description}
\item [{(i)}] $nh^{p/2}I_{1,1}=\delta_{n}^{2}nh^{p/2}o_{p}\left(1\right)+\delta_{n}\sqrt{n}h^{p/2}o_{p}\left(1\right)+o_{p}\left(1\right)$,
\item [{(ii)}] $nh^{p/2}I_{1,2}=\delta_{n}^{2}nh^{p/2}o_{p}\left(1\right)+\delta_{n}\sqrt{n}h^{p/2}o_{p}\left(1\right)+o_{p}\left(1\right)$,
\item [{(iii)}] $nh^{p/2}I_{1,4}=\delta_{n}^{2}nh^{p/2}o_{p}\left(1\right)+\delta_{n}\sqrt{n}h^{p/2}o_{p}\left(1\right)+\left(ng^{p}\right)^{-1/2}o_{p}\left(1\right)+o_{p}\left(1\right)$,
\item [{(iv)}] $nh^{p/2}I_{2,2}=\delta_{n}^{2}nh^{p/2}o_{p}\left(1\right)+\delta_{n}\sqrt{n}h^{p/2}o_{p}\left(1\right)+o_{p}\left(1\right)$.
\end{description}
\end{prop}
The proofs of the above propositions follow the ones in
\cite{Lavergne2000}). For illustration, we provide in the
Appendix the proofs of the first statements of each
proposition.

\paragraph{Case $I_n=\widehat{I}_{n}$.}
We have the following decomposition
\begin{equation}\label{ilvifl}
n^{\left(4\right)}\tilde{I}_{n} =
n\left(n-1\right)^{3}\widehat{I}_{n}-n^{\left(3\right)}V_{1n}-2n^{\left(3\right)}V_{2n}+n^{\left(2\right)}V_{3n}
\end{equation}
\begin{eqnarray*}
\mbox{where } \quad
V_{1n}& = &\dfrac{1}{n^{\left(3\right)}}\sum_{a}\left(Y_{i}-Y_{k}\right)\left(Y_{j}-Y_{k}\right)L_{nik}L_{njk}K_{nij}\psi_{ij}\, ,\\
V_{2n}& = &\dfrac{1}{n^{\left(3\right)}}\sum_{a}\left(Y_{i}-Y_{j}\right)\left(Y_{j}-Y_{k}\right)L_{nij}L_{njk}K_{nij}\psi_{ij}\, ,
\\
\mbox{and } \quad
V_{3n} &  = &
\dfrac{1}{n^{\left(2\right)}}\sum_{a}\left(Y_{i}-Y_{j}\right)^{2}L_{nij}^{2}K_{nij}\psi_{ij}
\, .
\end{eqnarray*}
Hence, to show that $\widehat{I}_{n}$ has the same asymptotic
distribution as $\tilde{I}_{n}$, it is sufficient to investigate the
behavior of $V_{1n}$ to $V_{3n}.$ Using $Y_i = r_i + u_i,$ it is
straightforward to see that the dominating terms in $V_{1n},V_{2n}$
and $V_{3n}$ are
$$V_{13}=\dfrac{1}{n^{\left(3\right)}}\sum_{a}\left(r_{i}-r_{k}\right)\left(r_{j}-r_{k}\right)L_{nik}L_{njk}K_{nij}\psi_{ij},$$
$$V_{23}=\dfrac{1}{n^{\left(3\right)}}\sum_{a}\left(r_{i}-r_{j}\right)\left(r_{j}-r_{k}\right)L_{nij}L_{njk}K_{nij}\psi_{ij},
\quad V_{33}=\dfrac{1}{n^{\left(2\right)}}\sum_{a}\left(r_{i}-r_{j}\right)^{2}L^2_{nij}K_{nij}\psi_{ij},$$
respectively.  Now
\begin{eqnarray*}
\mathbb{E}\left[|V_{13}|\right] & = &
\mathbb{E}\left[|\left(r_{i}-r_{k}\right)\left(r_{j}-r_{k}\right)L_{nik}L_{njk}K_{nij}|
  \right]\\ & = &
O\left(g^{-p}\right)\mathbb{E}\left[\left|r_{i}-r_{k}\right|\mathbf{L}_{nik}\mathbb{E}\left[\left|r_{j}-r_{k}\right|\mathbf{K}_{nij}\mid
    Z_{i},Z_{k}\right]\right] = O\left(g^{-p}\right) \, ,
\\ \mathbb{E}\left[|V_{23}|\right] & = &
\mathbb{E}\left[|\left(r_{i}-r_{j}\right)\left(r_{j}-r_{k}\right)L_{nij}L_{njk}K_{nij}|\right]
\\ &= &
\mathbb{E}\left[\mathbb{E}\left[\left|r_{j}-r_{k}\right|\mathbf{L}_{njk}\mid
    Z_{j}\right]\left|r_{i}-r_{j}\right|\mathbf{L}_{nij}\mathbf{K}_{nij}\right]\\ &
= &
o\left(1\right)\mathbb{E}\left[\left|r_{i}-r_{j}\right|\mathbf{L}_{nij}\mathbf{K}_{nij}\right]
= o\left(g^{-p}\right)
\\
\mathbb{E}\left[|V_{33}|\right] & = &
\mathbb{E}\left[\left(r_{i}-r_{j}\right)^{2}L_{nij}^{2}|K_{nij}|\right]\\ &
= &
O\left(g^{-2p}\right)\mathbb{E}\left[\left(r_{i}-r_{j}\right)^{2}\mathbf{K}_{nij}\right]
= o\left(g^{-2p}\right)
\, .
\end{eqnarray*}
It then follows that
$nh^{p/2}\left(\tilde{I}_{n}-\widehat{I}_{n}\right)=O_p\left(h^{p/2}g^{-p}\right)$
which is negligible if $h/g^{2}\to0$.  The asymptotic irrelevance of
the above diagonal terms thus require more restrictive relationships
between the bandwidths $h$ and $g$. For the sake of comparison, recall
that \cite{FanLi1996} impose $h^{(p+q)}g^{-2p}\to0$ while
\cite{Lavergne2000} require only $h^{p+q}g^{-p}\to0$. Since we do not
smooth the covariates $X$, we are able to further relax the
restriction between the two bandwidths.

\subsection{Proof of Corollary \ref{corr_test_o}}

It suffices to prove $\omega_{n}^{2}-\omega^{2} = o_p(1)$ with
$\omega_{n}^{2}$ any of $\widehat{\omega}_{n}^{2}$ or
$\tilde{\omega}_{n}^{2}$.  First we consider the case
$\omega_{n}^{2}=\widehat{\omega}_{n}^{2}.$ A direct approach would
consist in replacing the definition of $\hat u_i \hat f_i$ and $\hat
u_j \hat f_j$, writing $\widehat{\omega}_{n}^{2}$ as a $U-$statistic
of order 6, and studying its mean and variance.  {A shorter approach
  is based on empirical process tools. The price to pay is the
  stronger condition $n^{7/8}g ^p/\ln n \rightarrow \infty$ instead of
  $ng ^p \rightarrow \infty.$} Let $\Delta \hat f_i = \hat f_i - f_i$,
$\Delta \hat r_i \hat f_i = \hat r_i \hat f_i- r_i f_i$, and write
\begin{equation}\label{uifi_unif}
\hat u_i \hat f_i = u_if_i + Y_i \Delta \hat f_i  - \Delta \hat r_i\hat f_i.
\end{equation}
\begin{lem}\label{unif_omeg}
Under Assumption \ref{Sample}, if $r(\cdot)f(\cdot) \in
\mathcal{U}^{p},$  $L(\cdot)$ is a function of bounded
variation, {$g\rightarrow 0,$ and $n^{7/8}g ^p/\ln n \rightarrow \infty,$} then
$$
\sup_{1\leq i\leq n}\{|\Delta \hat r_i\hat f_i| + |\Delta \hat f_i|  \} = o_p(1).
$$
\end{lem}
The proof relies on the uniform convergence of empirical processes and
is provided in the Appendix. Now proceed as follows:
square Equation (\ref{uifi_unif}), replace $\hat u_i^2 \hat f_i^2$ in
the definition of $\widehat{\omega}_{n}^{2},$ and use Lemma
\ref{unif_omeg} to  deduce that
\[
\widehat{\omega}_{n}^{2} =
\dfrac{2h^{p}}{n^{\left(2\right)}}\sum_{a\left(2\right)} u_{i}^{2}
f_{i}^{2} u_{j}^{2} f_{j}^{2}K_{nij}^{2}\psi_{ij}^{2} + o_p(1)
\, .
\]
Elementary calculations of mean and variance yield
$$ \dfrac{2h^p}{n^{(2)}}\sum_{a\left(2\right)} u_{i}^{2}
  f_{i}^{2} u_{j}^{2} f_{j}^{2}K_{nij}^{2}\psi_{ij}^{2} - \omega^{2} =
  o_p(1),
$$
and thus $\widehat{\omega}_{n}^{2} - \omega^{2} = o_p(1).$

To deal with $\tilde{\omega}_{n}^{2}$, note that
$\tilde{\omega}_{n}^{2} - \widehat{\omega}_{n}^{2}$ consists of
``diagonal'' terms plus a term which is
$O\left(n^{-1}\tilde{\omega}_{n}^{2}\right)$.  By tedious but rather
straightforward calculations, one can check that such diagonal terms
are each of the form $n^{-1}g^{-p}$ times a $U-$statistic which is
bounded in probability. Hence $ \tilde{\omega}_{n}^{2} -
\widehat{\omega}_{n}^{2}= o_p(1)$.

\subsection{Proof of Theorem \ref{Bootstrap Consistency}}

Let $\overline{Z}$ denote the sample $(Y_i,W_i,X_i),$ $1\leq i\leq n.$
Since the limit distribution is continuous, it suffices to prove the
result pointwise by Polya's theorem. Hence we show that $\forall
t\in\mathbb{R}$, $\mathbb{P}\left[nh^{p/2}I_{n}^{*}/\omega_{n}^{*}\leq
  t\mid\overline{Z}\right]-\Phi\left(t\right)=o_p(1)$.

First, we consider the case $I_n^{*}=\tilde{I}_n$. Consider
\begin{eqnarray*}
I_{n,LV}^{*} & = & \dfrac{1}{n^{\left(4\right)}}\sum_{a}\left(\eta_{i}\hat{u}_{i}-\eta_{k}\hat{u}_{k}\right)\left(\eta_{j}\hat{u}_{j}-\eta_{l}\hat{u}_{l}\right)L_{nik}L_{njl}K_{nij}\psi_{ij}\\
 &  & +\dfrac{2}{n^{\left(4\right)}}\sum_{a}\left(\eta_{i}\hat{u}_{i}-\eta_{k}\hat{u}_{k}\right)\left(\hat{r}_{j}-\hat{r}_{l}\right)L_{nik}L_{njl}K_{nij}\psi_{ij}\\
 &  & +\dfrac{1}{n^{\left(4\right)}}\sum_{a}\left(\hat{r}_{i}-\hat{r}_{k}\right)\left(\hat{r}_{j}-r_{l}\right)L_{nik}L_{njl}K_{nij}\psi_{ij}\\
 & = & I_{1}^{*}+2I_{2}^{*}+I_{3}^{*}
\end{eqnarray*}
where we can further decompose
\begin{eqnarray*}
I_{1}^{*} & = & \dfrac{1}{n^{\left(4\right)}}\sum_{a}\eta_{i}\hat{u}_{i}\eta_{j}\hat{u}_{j}L_{nik}L_{njl}K_{nij}\psi_{ij}\\
 &  & -\dfrac{2}{n^{\left(4\right)}}\sum_{a}\eta_{j}\hat{u}_{j}\eta_{k}\hat{u}_{k}L_{nik}L_{njl}K_{nij}\psi_{ij}\\
 &  & +\dfrac{1}{n^{\left(4\right)}}\sum_{a}\eta_{k}\hat{u}_{k}\eta_{l}\hat{u}_{l}L_{nik}L_{njl}K_{nij}\psi_{ij}\\
 & = & I_{1,1}^{*}+I_{1,2}^{*}+I_{1,3}^{*}
\end{eqnarray*}
with %
\begin{eqnarray*}
I_{1,1}^{*} & = & \dfrac{\left(n-1\right)^{2}}{\left(n-3\right)\left(n-4\right)}\times\dfrac{1}{n^{\left(2\right)}}\sum_{a}\eta_{i}\hat{u}_{i}\eta_{j}\hat{u}_{j}\hat{f}_{i}\hat{f}_{j}K_{nij}\psi_{ij}\\
 &  & -\dfrac{2}{n-4}\times\dfrac{1}{n^{\left(3\right)}}\sum_{a}\eta_{i}\hat{u}_{i}\eta_{j}\hat{u}_{j}L_{nik}L_{nij}K_{nij}\psi_{ij}\\
 &  & -\dfrac{1}{n-4}\times\dfrac{1}{n^{\left(3\right)}}\sum_{a}\eta_{i}\hat{u}_{i}\eta_{j}\hat{u}_{j}L_{nik}L_{njk}K_{nij}\psi_{ij}\\
 &  & -\dfrac{1}{\left(n-3\right)\left(n-4\right)}\times\dfrac{1}{n^{\left(2\right)}}\sum_{a}\eta_{i}\hat{u}_{i}\eta_{j}\hat{u}_{j}L_{nij}^{2}K_{nij}\psi_{ij}\\
 & = & I_{0n}^{*}-\dfrac{2}{n-4}I_{1,1,1}^{*}-\dfrac{1}{n-4}I_{1,1,2}^{*}-\dfrac{1}{\left(n-3\right)\left(n-4\right)}I_{1,1,3}^{*}.
\end{eqnarray*}
Now let $D_{n}^{*}=\tilde{I}_{n}^{*}-I_{0n}^{*}$ and write
\begin{eqnarray*}
\Pr\left(\dfrac{nh^{p/2}\tilde{I}_{n}^{*}}{\tilde{\omega}_{n}^{*}}
\leq t\mid\overline{Z}\right) & = &
\!\!\Pr\left(\dfrac{nh^{p/2}\left(I_{0n}^{*}+D_{n}^{*}\right)}{\tilde{\omega}_{n}^{*}}
\leq t\mid\overline{Z}\right)\\
 & = &
\!\!\Pr\!\!\left(\dfrac{nh^{p/2}I_{0n}^{*}}{\widehat{\omega}_{n}}
+\dfrac{nh^{p/2}D_{n}^{*}}{\widehat{\omega}_{n}}
+\dfrac{nh^{p/2}\left(I_{0n}^{*}+D_{n}^{*}\right)}{\widehat{\omega}_{n}}
\left(\dfrac{\tilde{\omega}_{n}}{\widehat{\omega}_{n}^{*}}-1\right)
\leq t\mid\overline{Z}\right).
\end{eqnarray*}
It thus  suffices to prove that
$$
\Pr\left(\dfrac{nh^{p/2}I_{0n}^{*}}{\hat{\omega}_{n,FL}}\leq
t\mid\overline{Z}\right)-\Phi\left(t\right)\xrightarrow{p}0 \qquad \forall t \in \mathbb{R}\, ,$$
\begin{equation}\label{r_star_1}
\dfrac{nh^{p/2}D_{n}^{*}}{\hat{\omega}_{n,FL}}=o_p(1)\, ,\qquad \text{ and } \qquad  \dfrac{nh^{p/2}\left(I_{0n}^{*}+D_{n}^{*}\right)}{\hat{\omega}_{n,FL}}\left(\dfrac{\hat{\omega}_{n,FL}}{\hat{\omega}_{n,LV}^{*}}-1\right) =o_{p}\left(1\right).
\end{equation}
The first result is stated below.
\begin{prop}
Under the conditions of Theorem \ref{Bootstrap Consistency}, conditionally
on the observed sample, the statistic $nh^{p/2}I_{0n}^{*}/\hat{\omega}_{n,FL}$ converges in law to a standard normal distribution.
\label{BootstrapNormality}
\end{prop}
\begin{proof}
We proceed as in the proof of Proposition \ref{Normality} and check the
conditions for a CLT for martingale arrays, see Corollary 3.1 of
\cite{Hall1980}.  Define the martingale array $\left\{
S_{n,m}^{*},\,\mathcal{F}_{n,m}^{*},\,1\leq m\leq n,\, n\geq1\right\}
$ where $\mathcal{F}_{n,m}^{*}$ is the $\sigma$-field generated by
$\left\{ \overline{Z},\,\eta_{1},\,\dots,\eta_{m}\right\} $,
$S_{n,1}^{*}=0$, and $S_{n,m}^{*}=\sum_{i=1}^{m}G_{n,i}^{*}$ with
$$G_{n,i}^{*}=\dfrac{2h^{p/2}}{n-1}\eta_{i}\hat{u}_{i}\sum_{j=1}^{i-1}\eta_{j}\hat{u}_{j}\hat{f}_{i}\hat{f}_{j}K_{nij}\psi_{ij}
\, .
$$
Then
\[
I_{0n}^{*}  =
 \dfrac{\left(n-1\right)^{2}}{\left(n-3\right)\left(n-4\right)}\times\dfrac{1}{n^{\left(2\right)}}\sum_{a}\eta_{i}\hat{u}_{i}\eta_{j}\hat{u}_{j}\hat{f}_{i}\hat{f}_{j}K_{nij}\psi_{ij}
 =
\dfrac{\left(n-1\right)^{2}}{\left(n-3\right)\left(n-4\right)} S_{n,n}^{*}
\, .
\]
Now consider
\begin{eqnarray*}
V_{n}^{2*} & = & \sum_{i=2}^{n}\mathbb{E}\left[G_{n,i}^{2*}\mid\mathcal{F}_{n,i-1}^{*}\right]\\
 & = & \dfrac{4h^{p}}{\left(n-1\right)^{2}}\sum_{i=2}^{n}\sum_{j=1}^{i-1}\sum_{k=1}^{i-1}\hat{u}_{i}^{2}\eta_{j}\eta_{k}\hat{u}_{j}\hat{u}_{k}\hat{f}_{i}^{2}\hat{f}_{j}\hat{f}_{k}K_{nij}K_{nik}\psi_{ij}\psi_{ik}\\
 & = & \dfrac{4h^{p}}{\left(n-1\right)^{2}}\sum_{i=2}^{n}\sum_{j=1}^{i-1}\hat{u}_{i}^{2}\eta_{j}^{2}\hat{u}_{j}^{2}\hat{f}_{i}^{2}\hat{f}_{j}^{2}K_{nij}^{2}\psi_{ij}^{2}\\
 &  & +\dfrac{8h^{p}}{\left(n-1\right)^{2}}\sum_{i=3}^{n}\sum_{j=2}^{i-1}\sum_{k=1}^{j-1}\hat{u}_{i}^{2}\eta_{j}\eta_{k}\hat{u}_{j}\hat{u}_{k}\hat{f}_{i}^{2}\hat{f}_{j}\hat{f}_{k}K_{nij}K_{nik}\psi_{ij}\psi_{ik}\\
 & = & A_{n}^{*}+B_{n}^{*}.
\end{eqnarray*}
Note that $\mathbb{E}\left[A_{n}^{*}\mid\overline{Z}\right]=[n/(n-1)]\mathbb{E}\left[\widehat{\omega}_{n}^{2}\right]$
and that %
\begin{eqnarray*}
\mbox{Var}\left[\tilde{A}_{n}^{*}\mid\overline{Z}\right] & \leq & \dfrac{16h^{2p}\mathbb{E}\left[\eta^{4}\right]}{\left(n-1\right)^{4}}\sum_{i=2}^{n}\sum_{i^{\prime}=2}^{n}\sum_{j=1}^{i\wedge i^{\prime}-1}\hat{u}_{i}^{2}\hat{u}_{i^{\prime}}^{2}\hat{u}_{j}^{4}\hat{f}_{i}^{2}\hat{f}_{i^{\prime}}^{2}\hat{f}_{j}^{4}K_{nij}^{2}K_{ni^{\prime}j}^{2}\psi_{ij}^{2}\psi_{i^{\prime}j}^{2}\\
 & \leq & \dfrac{16h^{2p}\mathbb{E}\left[\eta^{4}\right]}{\left(n-1\right)^{4}}\sum_{i=2}^{n}\sum_{j=1}^{i-1}\hat{u}_{i}^{4}\hat{u}_{j}^{4}\hat{f}_{i}^{4}\hat{f}_{j}^{4}K_{nij}^{4}\psi_{ij}^{4}\\
 &  & +\dfrac{32h^{2p}\mathbb{E}\left[\eta^{4}\right]}{\left(n-1\right)^{4}}\sum_{i=3}^{n}\sum_{i^{\prime}=2}^{i-1}\sum_{j=1}^{i^{\prime}-1}\hat{u}_{i}^{2}\hat{u}_{i^{\prime}}^{2}\hat{u}_{j}^{4}\hat{f}_{i}^{2}\hat{f}_{i^{\prime}}^{2}\hat{f}_{j}^{4}K_{nij}^{2}K_{ni^{\prime}j}^{2}\psi_{ij}^{2}\psi_{i^{\prime}j}^{2}\\
 &=& Q_{1n}+Q_{2n}.
\end{eqnarray*}
On the other hand,
\begin{eqnarray*}
\mathbb{E}\left[B_{n}^{*2}\mid\overline{Z}\right] & = & \dfrac{64h^{2p}}{\left(n-1\right)^{4}}\sum_{i=3}^{n}\sum_{i^{\prime}=3}^{n}\sum_{j=2}^{i\wedge i^{\prime}-1}\sum_{k=1}^{j-1}\hat{u}_{i}^{2}\hat{u}_{i^{\prime}}^{2}\hat{u}_{j}^{2}\hat{u}_{k}^{2}\hat{f}_{i}^{2}\hat{f}_{i^{\prime}}^{2}\hat{f}_{j}^{2}\hat{f}_{k}^{2}K_{nij}K_{ni^{\prime}j}K_{nik}K_{ni^{\prime}k}\psi_{ij}\psi_{i^{\prime}j}\psi_{ik}\psi_{i^{\prime}k}\\
 & = & \dfrac{64h^{2p}}{\left(n-1\right)^{4}}\sum_{i=3}^{n}\sum_{j=2}^{i-1}\sum_{k=1}^{j-1}\hat{u}_{i}^{4}\hat{u}_{j}^{2}\hat{u}_{k}^{2}\hat{f}_{i}^{4}\hat{f}_{j}^{2}\hat{f}_{k}^{2}K_{nij}^{2}K_{nik}^{2}\psi_{ij}^{2}\psi_{ik}^{2}\\
 &  & +\dfrac{128h^{2p}}{\left(n-1\right)^{4}}\sum_{i=4}^{n}\sum_{i^{\prime}=3}^{i-1}\sum_{j=2}^{i^{\prime}-1}\sum_{k=1}^{j-1}\hat{u}_{i}^{2}\hat{u}_{i^{\prime}}^{2}\hat{u}_{j}^{2}\hat{u}_{k}^{2}\hat{f}_{i}^{2}\hat{f}_{i^{\prime}}^{2}\hat{f}_{j}^{2}\hat{f}_{k}^{2}K_{nij}K_{ni^{\prime}j}K_{nik}K_{ni^{\prime}k}\psi_{ij}\psi_{i^{\prime}j}\psi_{ik}\psi_{i^{\prime}k}
 \\
 &=& Q_{3n}+Q_{4n}.
\end{eqnarray*}
Finally the Lindeberg condition involves
\begin{align*}
 & \sum_{i=1}^{n}\mathbb{E}\left[G_{n,i}^{2*}I\left(\left|G_{n,i}^{*}\right|>\varepsilon\right)\mid\mathcal{F}_{n,i-1}^{*}\right]\\
\leq & \dfrac{1}{\varepsilon^{4}}\sum_{i=1}^{n}\mathbb{E}\left[G_{n,i}^{4*}\mid\mathcal{F}_{n,i-1}^{*}\right]\\
\leq & \dfrac{16h^{2p}\mathbb{E}\left[\eta^{4}\right]}{\varepsilon^{4}\left(n-1\right)^{4}}\sum_{i=2}^{n}\sum_{j=1}^{i-1}\sum_{k=1}^{i-1}\hat{u}_{i}^{4}\hat{u}_{j}^{2}\hat{u}_{k}^{2}\hat{f}_{i}^{4}\hat{f}_{j}^{2}\hat{f}_{k}^{2}K_{nij}^{2}K_{nik}^{2}\psi_{ij}^{2}\psi_{ik}^{2}\\
\leq & \dfrac{16h^{2p}\mathbb{E}\left[\eta^{4}\right]}{\varepsilon^{4}\left(n-1\right)^{4}}\sum_{i=2}^{n}\sum_{j=1}^{i-1}\hat{u}_{i}^{4}\hat{u}_{j}^{4}\hat{f}_{i}^{4}\hat{f}_{j}^{4}K_{nij}^{4}\psi_{ij}^{4}\\
 & +\dfrac{32h^{2p}\mathbb{E}\left[\eta^{4}\right]}{\varepsilon^{4}\left(n-1\right)^{4}}\sum_{i=3}^{n}\sum_{j=2}^{i-1}\sum_{k=1}^{j-1}\hat{u}_{i}^{4}\hat{u}_{j}^{2}\hat{u}_{k}^{2}\hat{f}_{i}^{4}\hat{f}_{j}^{2}\hat{f}_{k}^{2}K_{nij}^{2}K_{nik}^{2}\psi_{ij}^{2}\psi_{ik}^{2}\\
 &= Q_{5n}+Q_{6n}.
\end{align*}
It thus suffices to show that $Q_{jn}= o_{p}(1)$, $j=1,\ldots 6$.
Now, there exist positive random variables $\tilde{\gamma}_{1n}$ and
$\tilde{\gamma}_{2n}$ such that
$\tilde{\gamma}_{1n}+\tilde{\gamma}_{2n}=o_{p}\left(1\right)$ and
$$
\hat{u}_{i}^{2k}\hat{f}_{i}^{2k}\leq
3^{2k-1} \left(u_{i}^{2k}f_{i}^{2k}+Y_{i}^{2k}f_{i}^{2k}
\tilde{\gamma}_{1n}^{2k}+\tilde{\gamma}_{2n}^{2k}\right)
\qquad \forall 1\leq i\leq n \quad \mbox{and  } \quad \forall k = 1,2
\in\left\{ 1,2\right\}
\, .
$$
Indeed,
$
\hat{u}_{i}\hat{f}_{i} =
u_{i}f_{i}+Y_{i}f_{i}f_{i}^{-1}\left(\hat{f}_{i}-f_{i}\right)
+\left[\hat{r}_{i}\hat{f}_{i}-r_{i}f_{i}\right]
=  u_{i}f_{i}+Y_{i}f_{i}\gamma_{1i}-\gamma_{2i}
$,
where $\sup_{1\leq i\leq
  n}\left|\gamma_{ji}\right|\leq\tilde{\gamma_{j}}$ and
$\tilde{\gamma_{j}}=o_{p}\left(1\right)$ by Lemma \ref{unif_omeg}.
Hence
\[
\hat{u}_{i}^{2}\hat{f}_{i}^{2}\leq3\left(u_{i}^{2}f_{i}^{2}+Y_{i}^{2}f_{i}^{2}
\tilde{\gamma}_{1n}^{2}+\tilde{\gamma}_{2n}^{2}\right)
\, .
\]
The inequality for $k=2$ is obtained similarly.  Using these
inequalities, one can bound the expectations of $|Q_{1n}|$ to
$|Q_{6n}|$ and thus show that $|Q_{1n}|+\cdots+|Q_{6n}|=o_p(1)$.
\end{proof}

Next we show (\ref{r_star_1}). First we need the following.
\begin{prop}
Under the conditions of Theorem \ref{Bootstrap Consistency},
$\dfrac{\hat{\omega}_{n,FL}}{\hat{\omega}_{n,FL}^{*}}\xrightarrow{p}1
$ and
$
\dfrac{\hat{\omega}_{n,FL}}{\hat{\omega}_{n,LV}^{*}}\xrightarrow{p}1
$.
\label{VarianceBootstrap}\end{prop}
The proof uses the following result, which is proved in the Appendix.
\begin{lem}\label{DeltaUiStar}
Under the conditions of Theorem \ref{Bootstrap Consistency},
$\sup_{1\leq i\leq n} |\hat{u}_{i}^{*}\hat{f}_{i}-u_{i}^{*}\hat{f}_{i}| = o_{p}\left(1\right)$,
where
$u_{i}^{*} =\eta_{i} \widehat{u}_{i}$ and
\[
\hat{u}_{i}^{*}  =  Y_{i}^{*}-\dfrac{\sum_{k\neq i}Y_{k}^{*}L_{nik}}{\sum_{k\neq i}L_{nik}}
\, .
\]
\end{lem}
\begin{proof}
Using Lemma \ref{DeltaUiStar}, we have
\[
\hat{\omega}_{n,FL}^{*2}=\omega_{n}^{*2}+o_{p}\left(1\right)
\]
where $\omega_{n}^{*2}=\dfrac{2h^{p}}{n^{\left(2\right)}}\sum_{a}u_{i}^{*2}u_{j}^{*2}\hat{f}_{i}^{2}\hat{f}_{j}^{2}K_{nij}^{2}\psi_{ij}^{2}$.
Notice that $\mathbb{E}\left[\omega_{n}^{*2}\mid\overline{Z}\right]=\hat{\omega}_{n,FL}^{2}$
and that
$$
\mbox{Var}\left(\omega_{n}^{*2}-\hat{\omega}_{n,FL}^{2}\right) = \mbox{Var}\left(\mathbb{E}\left[\omega_{n}^{*2}-\hat{\omega}_{n,FL}^{2}\mid\overline{Z}\right]\right)
+\mathbb{E}\left[\mbox{Var}\left(\omega_{n}^{*2}\mid\overline{Z}\right)\right]
$$
where the first term is zero and $$\mbox{Var}\left(\omega_{n}^{*2}\mid\overline{Z}\right)=\dfrac{8h^{2p}\mbox{Var}\left(\eta^{2}\right)}{\left\{n^{\left(2\right)}\right\}^{2}}\sum_{a}\hat{u}_{i}^{4}\hat{u}_{j}^{4}\hat{f}_{i}^{4}\hat{f}_{j}^{4}K_{nij}^{4}\psi_{ij}^{4}.$$
Then, $$\dfrac{\hat{\omega}_{n,FL}}{\hat{\omega}_{n,FL}^{*}}=1+\dfrac{\hat{\omega}_{n,FL}-\hat{\omega}_{n,FL}^{*}}{\hat{\omega}_{n,FL}^{*}}=1+\dfrac{o_{p}\left(1\right)}{\omega^{2}\left[1+o_{p}\left(1\right)\right]} = 1+o_p(1).$$
Since $\hat{\omega}_{n,LV}^{*}-\hat{\omega}_{n,FL}^{*}$ contains only diagonal terms,  we  deduce that $\hat{\omega}_{n,FL}/\hat{\omega}_{n,LV}^{*}\xrightarrow{p}1.$
\end{proof}
We next have to bound $D_{n}^{*}=I_{n,LV}^{*}-I_{0n}^{*}.$ For this,
let us decompose $$\hat{r}_{i}-\hat{r}_{k}=\left(\hat{r}_{i}-r_{i}\right)-\left(\hat{r}_{k}-r_{k}\right)+\left(r_{i}-r_{k}\right)$$
and replace all such differences appearing in the definition of $D_{n}^{*}$.
First, let us look at $I_{3}^{*}$ which does not contain any bootstrap variable $\eta.$ We obtain
\begin{eqnarray*}
I_{3}^{*} & = & \dfrac{1}{n^{\left(4\right)}}\sum_{a}\left(\hat{r}_{i}-\hat{r}_{k}\right)\left(\hat{r}_{j}-\hat{r}_{l}\right)L_{nik}L_{njl}K_{nij}\psi_{ij}\\
 & = & \dfrac{1}{n^{\left(4\right)}}\sum_{a}\left(r_{i}-r_{k}\right)\left(r_{j}-r_{l}\right)L_{nik}L_{njl}K_{nij}\psi_{ij} \\
& & +\dfrac{1}{n^{\left(4\right)}}\sum_{a}\left(\hat{r}_{i}-r_{i}\right)\left(\hat{r}_{j}-r_{j}\right)L_{nik}L_{njl}K_{nij}\psi_{ij}\\
 &  & +\dfrac{1}{n^{\left(4\right)}}\sum_{a}\left(\hat{r}_{k}-r_{k}\right)\left(\hat{r}_{l}-r_{l}\right)L_{nik}L_{njl}K_{nij}\psi_{ij}\\
& & +\dfrac{2}{n^{\left(4\right)}}\sum_{a}\left(\hat{r}_{i}-r_{i}\right)\left(r_{j}-r_{l}\right)L_{nik}L_{njl}K_{nij}\psi_{ij}\\
 &  & -\dfrac{2}{n^{\left(4\right)}}\sum_{a}\left(\hat{r}_{k}-r_{k}\right)\left(r_{j}-r_{l}\right)L_{nik}L_{njl}K_{nij}\psi_{ij}\\
& &  -\dfrac{2}{n^{\left(4\right)}}\sum_{a}\left(\hat{r}_{k}-r_{k}\right)\left(\hat{r}_{j}-r_{j}\right)L_{nik}L_{njl}K_{nij}\psi_{ij}\\
 & = & I_{3,1}^{*}+I_{3,2}^{*}+I_{3,3}^{*}+2I_{3,4}^{*}-2I_{3,5}^{*}-2I_{3,6}^{*}.
\end{eqnarray*}
Next, use the fact that
\begin{eqnarray}\label{dsqa}
\hat{r}_{i}-r_{i} & = & \left(n-1\right)^{-1}\hat{f}_{i}^{-1}\sum_{i^{\prime}\neq i}\left(Y_{i^{\prime}}-r_{i}\right)L_{nii^{\prime}}\nonumber \\
 & = & \left(n-1\right)^{-1}\hat{f}_{i}^{-1}\sum_{i^{\prime}\neq i}\left(r_{i^{\prime}}-r_{i}\right)L_{nii^{\prime}}+\left(n-1\right)^{-1}\hat{f}_{i}^{-1}\sum_{i^{\prime}\neq i}u_{i^{\prime}}L_{nii^{\prime}}
\end{eqnarray}
and further replace terms like $\hat{r}_{i}-r_{i}.$
Among the terms $I_{3,1}^{*}$ to $I_{3,6}^{*},$ the term $I_{3,1}^{*}$ could be easily handled with existing results in \cite{Lavergne2000}. Namely $nh^{p/2}I_{3,1}^{*}=nh^{p/2}O_p\left(g^{2s}\right)+o_p\left(1\right)$ by Proposition 7 of \cite{Lavergne2000}. For the other five terms we have to control the density estimates appearing in the denominators.
%
For this purpose, let us introduce the notation $\Delta\left(f_{i}^{I}\right)^{-1}=\left(\hat{f}_{i}^{I}\right)^{-1}-f_{i}^{-1}$
and write
\begin{equation}\label{zzae}
\dfrac{n-\left|I\right|}{n-1}\times\hat{f}_{i}^{-1} = \left(\dfrac{\left(n-\left|I\right|\right)\hat{f}_{i}^{I}}{\left(n-1\right)\hat{f}_{i}}-1\right)\left(\hat{f}_{i}^{I}\right)^{-1}+\left(\hat{f}_{i}^{I}\right)^{-1}
= \dfrac{\sum_{k\in I}L_{nik}}{\left(n-1\right)\hat{f}_{i}\hat{f}_{i}^{I}}+\Delta\left(f_{i}^{I}\right)^{-1}+f_{i}^{-1}.
\end{equation}
Then, we obtain for instance
\begin{eqnarray*}
 I_{3,5}^{*}&=  & \dfrac{1}{n^{\left(4\right)}}\sum_{a}\left(\hat{r}_{k}-r_{k}\right)\left(r_{j}-r_{l}\right)L_{nik}L_{njl}K_{nij}\psi_{ij}\\
 & = & \dfrac{1}{n^{\left(5\right)}}\sum_{a\left(4\right)}\sum_{k^{\prime}\neq k}f_{k}^{-1}\left(r_{k^{\prime}}-r_{k}\right)\left(r_{j}-r_{l}\right)L_{nkk^{\prime}}L_{nik}L_{njl}K_{nij}\psi_{ij}\\
 &  & +\dfrac{1}{n^{\left(5\right)}}\sum_{a\left(4\right)}\sum_{k^{\prime}\neq k}\Delta\left(f_{k}^{i,j,l,k^{\prime}}\right)^{-1}\left(r_{k^{\prime}}-r_{k}\right)\left(r_{j}-r_{l}\right)L_{nkk^{\prime}}L_{nik}L_{njl}K_{nij}\psi_{ij}\\
 &  & +\dfrac{1}{\left(n-1\right)n^{\left(5\right)}}\sum_{a\left(4\right)}\sum_{k^{\prime}\neq k}\left(\hat{f}_{k}\hat{f}_{k}^{i,j,l,k^{\prime}}\right)^{-1}\left(L_{nik}+L_{njk}+L_{nlk}+L_{nk^{\prime}k}\right)\\
 && \qquad \qquad \qquad \qquad \qquad \qquad \times \left(r_{k^{\prime}}-r_{k}\right)\left(r_{j}-r_{l}\right)L_{nkk^{\prime}}L_{nik}L_{njl}K_{nij}\psi_{ij}\\
 &  & +\dfrac{1}{n^{\left(5\right)}}\sum_{a\left(4\right)}\sum_{k^{\prime}\neq k}f_{k}^{-1}u_{k^{\prime}}\left(r_{j}-r_{l}\right)L_{nkk^{\prime}}L_{nik}L_{njl}K_{nij}\psi_{ij}\\
 &  & +\dfrac{1}{n^{\left(5\right)}}\sum_{a\left(4\right)}\sum_{k^{\prime}\neq k}\Delta f_{k}^{-1}u_{k^{\prime}}\left(r_{j}-r_{l}\right)L_{nkk^{\prime}}L_{nik}L_{njl}K_{nij}\psi_{ij}\\
 &  & +\dfrac{1}{\left(n-1\right)n^{\left(5\right)}}\sum_{a\left(4\right)}\sum_{k^{\prime}\neq k}\left(\hat{f}_{k}\hat{f}_{k}^{i,j,l,k^{\prime}}\right)^{-1}\left(L_{nik}+L_{njk}+L_{nlk}+L_{nk^{\prime}k}\right)\\
 && \qquad \qquad \qquad \qquad \qquad \qquad \times u_{k^{\prime}}\left(r_{j}-r_{l}\right)L_{nkk^{\prime}}L_{nik}L_{njl}K_{nij}\psi_{ij}\\
 & = & I_{3,5,1}^{*}+I_{3,5,2}^{*}+I_{3,5,3}^{*}+I_{3,5,4}^{*}+I_{3,5,5}^{*}+I_{3,5,6}^{*}.
\end{eqnarray*}
Next, if we consider for instance $I_{3,5,1}^{*}$ that contains only terms like $f_{i}^{-1}$ appearing from the decomposition \ref{zzae}, we obtain
\begin{eqnarray*}
I_{3,5,1}^{*} & = & \dfrac{1}{n^{\left(5\right)}}\sum_{a\left(5\right)}f_{k}^{-1}\left(r_{k^{\prime}}-r_{k}\right)\left(r_{j}-r_{l}\right)L_{nkk^{\prime}}L_{nik}L_{njl}K_{nij}\psi_{ij}\\
 &  & +\dfrac{1}{n^{\left(5\right)}}\sum_{a\left(4\right)}f_{k}^{-1}\left(r_{i}-r_{k}\right)\left(r_{j}-r_{l}\right)L_{nik}^{2}L_{njl}K_{nij}\psi_{ij}\\
 &  & +\dfrac{1}{n^{\left(5\right)}}\sum_{a\left(4\right)}f_{k}^{-1}\left(r_{j}-r_{k}\right)\left(r_{j}-r_{l}\right)L_{njk}L_{nik}L_{njl}K_{nij}\psi_{ij}\\
 &  & +\dfrac{1}{n^{\left(5\right)}}\sum_{a\left(4\right)}f_{k}^{-1}\left(r_{l}-r_{k}\right)\left(r_{j}-r_{l}\right)L_{nlk}L_{nik}L_{njl}K_{nij}\psi_{ij}\\
 & = & I_{3,5,1,1}^{*}+I_{3,5,1,2}^{*}+I_{3,5,1,3}^{*}+I_{3,5,1,4}^{*}
\end{eqnarray*}
where the terms $I_{3,5,1,2}^{*}$ to $I_{3,5,1,4}^{*}$ are  called ``diagonal terms''. Such terms require more restrictions on the bandwidths.
next, the terms with containing terms like  $\Delta\left(f_{i}^{I}\right)^{-1}$ produced by the decomposition (\ref{zzae})
can be treated like in the Propositions 8 to 11 of Lavergne et Vuong (2000).
Finally, given that $I$ is finite and with fixed cardinal
$$\left(n-1\right)^{-1}\hat{f}_{i}^{-1}\left(\hat{f}_{i}^{I}\right)^{-1}\sum_{k\in I}L_{nik}=O_{p}\left(n^{-1}g^{-p}\right)=o_p(1)$$ given that $\left\Vert f^{-1}\right\Vert _{\infty}<\infty.$ Therefore the terms of $I_{3}^{*}$ containing $\left(n-1\right)^{-1}\hat{f}_{i}^{-1}\left(\hat{f}_{i}^{I}\right)^{-1}\sum_{k\in I}L_{nik}$ can be easily handled by taking absolute values.
Now let us investigate the diagonal term $I_{3,5,1,2}^{*}$. We have
\begin{eqnarray}
\label{eq:diagonalTerm}
\mathbb{E}\left[|I_{3,5,1,2}^{*}|\right] \nonumber
 & = & O\left(n^{-1}\right)\mathbb{E}\left[f_{k}^{-1}\left|r_{j}-r_{k}\right|\left|r_{j}-r_{l}\right||L_{njk}||L_{nik}||L_{njl}||K_{nij}|\right]\nonumber\\
 & = & O\left(n^{-1}g^{-p}\right)\mathbb{E}\left[f_{k}^{-1}\left|r_{j}-r_{k}\right|\left|r_{j}-r_{l}\right||L_{njk}||L_{njl}||K_{nij}|\right]\nonumber\\
 & = & O\left(n^{-1}g^{-p}\right)\mathbb{E}\left[f_{k}^{-1}\left|r_{j}-r_{k}\right||L_{njk}|\mathbb{E}\left[\left|r_{j}-r_{l}\right||L_{njl}|\mid Z_{j}\right]|K_{nij}|\right]\nonumber\\
 & = & o\left(n^{-1}g^{-p}\right)\mathbb{E}\left[f_{k}^{-1}\left|r_{j}-r_{k}\right||L_{njk}||K_{nij}|\right]\nonumber\\
 & = & o\left(n^{-1}g^{-p}\right)\nonumber.
\end{eqnarray}
To prove that he term $I_{3,5,1,2}^{*}= o_p(nh^{p/2})$ it suffices to prove $\mathbb{E}\left[|I_{3,5,1,2}^{*}|\right]= o(nh^{p/2})$ and this latter rate  is implied by the condition $h/g^2 = o(1).$
This additional condition on the bandwidths is not surprising as the bootstrapped statistic introduced ``diagonal''
terms as in Fan et Li (1996) which indeed require the condition $h/g^{2}\to 0$.

Let us now consider a term in the decomposition of $D_{n}^{*}$ that
involve bootstrap variables $\eta$, namely we investigate $I_{2}^{*}.$
The arguments for the other terms are similar. Consider
\begin{eqnarray*}
I_{2}^{*} & = & \dfrac{1}{n^{\left(4\right)}}\sum_{a}\left(\eta_{i}\hat{u}_{i}-\eta_{k}\hat{u}_{k}\right)\left(\hat{r}_{j}-\hat{r}_{l}\right)L_{nik}L_{njl}K_{nij}\psi_{ij}\\
 & = & \dfrac{1}{n^{\left(4\right)}}\sum_{a}\eta_{i}\hat{u}_{i}\left(r_{j}-r_{l}\right)L_{nik}L_{njl}K_{nij}\psi_{ij}
 +\dfrac{1}{n^{\left(4\right)}}\sum_{a}\eta_{i}\hat{u}_{i}\left(\hat{r}_{j}-r_{j}\right)L_{nik}L_{njl}K_{nij}\psi_{ij}\\
 &  & -\dfrac{1}{n^{\left(4\right)}}\sum_{a}\eta_{i}\hat{u}_{i}\left(\hat{r}_{l}-r_{l}\right)L_{nik}L_{njl}K_{nij}\psi_{ij}
 -\dfrac{1}{n^{\left(4\right)}}\sum_{a}\eta_{k}\hat{u}_{k}\left(r_{j}-r_{l}\right)L_{nik}L_{njl}K_{nij}\psi_{ij}\\
 &  & -\dfrac{1}{n^{\left(4\right)}}\sum_{a}\eta_{k}\hat{u}_{k}\left(\hat{r}_{j}-r_{j}\right)L_{nik}L_{njl}K_{nij}\psi_{ij}
 +\dfrac{1}{n^{\left(4\right)}}\sum_{a}\eta_{k}\hat{u}_{k}\left(\hat{r}_{l}-r_{l}\right)L_{nik}L_{njl}K_{nij}\psi_{ij}\\
 & = & I_{2,1}^{*}+I_{2,2}^{*}-I_{2,3}^{*}-I_{2,4}^{*}-I_{2,5}^{*}+I_{2,6}^{*}.
\end{eqnarray*}
Next it suffices to use the fact that $$\hat{u}_{i}=u_{i}-\hat{f}_{i}^{-1}\sum_{i^{\prime}\neq i}u_{i^{\prime}}L_{nii^{\prime}}+\hat{f}_{i}^{-1}\sum_{i^{\prime}\neq i}\left(r_{i}-r_{i^{\prime}}\right)L_{nii^{\prime}}.$$ For instance, using this identity with $I_{2,1}^{*}$ we can write
\begin{eqnarray*}
I_{2,1}^{*} & = & \dfrac{1}{n^{\left(4\right)}}\sum_{a}\eta_{i}u_{i}\left(r_{j}-r_{l}\right)L_{nik}L_{njl}K_{nij}\psi_{ij}\\
 &  & -\dfrac{1}{\left(n-1\right)n^{\left(4\right)}}\sum_{a}\sum_{i^{\prime}\neq i}\hat{f}_{i}^{-1}\eta_{i}u_{i^{\prime}}\left(r_{j}-r_{l}\right)L_{nik}L_{njl}K_{nij}\psi_{ij}\\
 &  & +\dfrac{1}{\left(n-1\right)n^{\left(4\right)}}\sum_{a}\sum_{i^{\prime}\neq i}\hat{f}_{i}^{-1}\eta_{i}\left(r_{i}-r_{i^{\prime}}\right)\left(r_{j}-r_{l}\right)L_{nik}L_{njl}K_{nij}\psi_{ij}\\
 & = & \dfrac{1}{n^{\left(3\right)}}\sum_{a}\eta_{i}u_{i}f_{i}\left(r_{j}-r_{l}\right)L_{njl}K_{nij}\psi_{ij}\\
 &  & +\dfrac{1}{n^{\left(4\right)}}\sum_{a}\eta_{i}u_{i}\left(r_{j}-r_{l}\right)\Delta f_{i}^{j,l}L_{njl}K_{nij}\psi_{ij}\\
 &  & -\dfrac{1}{\left(n-1\right)n^{\left(4\right)}}\sum_{a}\sum_{i^{\prime}\neq i}f_{i}^{-1}\eta_{i}u_{i^{\prime}}\left(r_{j}-r_{l}\right)L_{nik}L_{njl}K_{nij}\psi_{ij}\\
 &  & -\dfrac{1}{n^{\left(5\right)}}\sum_{a}\Delta\left(f_{i}^{j,k,l,i^{\prime}}\right)^{-1}\eta_{i}u_{i^{\prime}}\left(r_{j}-r_{l}\right)L_{nik}L_{njl}K_{nij}\psi_{ij}\\
 &  & -\dfrac{1}{\left(n-1\right)n^{\left(4\right)}}\sum_{a}\sum_{i^{\prime}\neq i}\left(\hat{f}_{i}\hat{f}_{i}^{j,k,l,i^{\prime}}\right)^{-1}\eta_{i}u_{i^{\prime}}\left(r_{j}-r_{l}\right)L_{nik}L_{njl}K_{nij}\psi_{ij}\\
 &  & +\dfrac{1}{\left(n-1\right)n^{\left(4\right)}}\sum_{a}\sum_{i^{\prime}\neq i}f_{i}^{-1}\eta_{i}\left(r_{i}-r_{i^{\prime}}\right)\left(r_{j}-r_{l}\right)L_{nik}L_{njl}K_{nij}\psi_{ij}\\
 &  & +\dfrac{1}{n^{\left(5\right)}}\sum_{a}\Delta\left(f_{i}^{j,k,l,i^{\prime}}\right)^{-1}\eta_{i}\left(r_{i}-r_{i^{\prime}}\right)\left(r_{j}-r_{l}\right)L_{nik}L_{njl}K_{nij}\psi_{ij}\\
 &  & +\dfrac{1}{n^{\left(5\right)}}\sum_{a}\left(\hat{f}_{i}\hat{f}_{i}^{j,k,l,i^{\prime}}\right)^{-1}\eta_{i}\left(r_{i}-r_{i^{\prime}}\right)\left(r_{j}-r_{l}\right)L_{nik}L_{njl}K_{nij}\psi_{ij}\\
 & = & I_{2,1,1}^{*}+I_{2,1,2}^{*}+I_{2,1,3}^{*}+I_{2,1,4}^{*}+I_{2,1,5}^{*}+I_{2,1,6}^{*}+I_{2,1,7}^{*}+I_{2,1,8}^{*}
\end{eqnarray*}
Handling one problem at a time, let us notice that $I_{2,1,1}^{*}$
is a zero-mean $U-$statistic of order three with kernel $H_{n}\left(Z_{i}^{*},Z_{j}^{*},Z_{l}^{*}\right)=\eta_{i}u_{i}f_{i}\left(r_{j}-r_{l}\right)L_{njl}K_{nij}\psi_{ij}$ where $Z_{i}^{*}=\left(Y_{i},W_{i},X_{i},\eta_{i}\right)$. Using the Hoeffding decomposition of $I_{2,1,1}^{*}$ in degenerate $U-$statistics, it is easy to check that the third and second order projections are small. For the first order degenerate $U-$statistic it suffices to note that
$\mathbb{E}\left[H_{n}\mid Z_{j}^{*}\right]=\mathbb{E}\left[H_{n}\mid Z_{l}^{*}\right]=0$
and $\mathbb{E}\left[H_{n}\mid Z_{i}^{*}\right]=\eta_{i}u_{i}f_{i}\mathbb{E}\left[\left(r_{j}-r_{l}\right)L_{njl}K_{nij}\psi_{ij}\mid Z_{i}\right]$
so that
\begin{eqnarray*}
\mathbb{E}\left[\mathbb{E}^{2}\left[H_{n}\mid Z_{i}^{*}\right]\right] & = & \mathbb{E}\left[\eta_{i}^{2}u_{i}^{2}f_{i}^{2}\mathbb{E}^{2}\left[\left(r_{j}-r_{l}\right)L_{njl}K_{nij}\psi_{ij}\mid Z_{i}\right]\right]\\
 & = & \mathbb{E}\left[u_{i}^{2}f_{i}^{2}\mathbb{E}^{2}\left[\left(r_{j}-r_{l}\right)L_{njl}K_{nij}\psi_{ij}\mid Z_{i}\right]\right]
\end{eqnarray*}
which, given that $\left\Vert \psi\right\Vert _{\infty}<\infty,$  is similar to the term $\xi_1$ bounded in the proof of Proposition 5 of
Lavergne et Vuong (2000).

\bigskip

Finally, let us briefly consider the case $I_n^{*}=\tilde{I}_n.$ Like
in the decomposition (\ref{ilvifl}), we have
\[
n\left(n-1\right)^{3}I_{n,FL}^{*}=n^{\left(4\right)}I_{n,LV}^{*}+n^{\left(3\right)}V_{1n}^{*}+2n^{\left(3\right)}V_{2n}^{*}-n^{\left(2\right)}V_{3n}^{*}
\]
where $\forall j\in\left\{ 1,2,3\right\} $, the $V_{jn}^{*}$s are
obtained by replacing the $Y_{i}$s by the $Y_{i}^{*}$s in the
$V_{jn}$s. All these terms could be handled by arguments similar to
the ones detailed above.  The proof of Theorem \ref{Bootstrap
  Consistency} is now complete.

\small
\bibliographystyle{ecta}
\bibliography{Misspecification}
\normalsize

\vfill\eject

\section*{Appendix (not for publication)}
 We here provide proofs of technical lemmas and additional details for
 the proofs in the manuscript. We define $Z_{i} = \left( Y_{i}, W_{i},
 X_{i}\right)$, $\|\psi\|_\infty=\sup_{x\in\mathbb{R}^q} |\psi(x)|$,
$$
\mathbf{K}_{nij} = |K_{nij}| = \frac{1}{h^p}
 \left|K\left( \frac{W_i - W_j}{h}\right)\right|,
\qquad \mbox{and}\qquad
 \mathbf{L}_{nij} = |L_{nij}| = \frac{1}{g^p} \left|L\left( \frac{W_i
   - W_j}{g}\right)\right|
\, .
$$

\begin{proof}[Proof of  Lemma \ref{AAn}]
1.
We have
$$
\mathbb{E}\left[A_{n}\right] = \dfrac{4h^{p}}{\left(n-1\right)^{2}}\sum_{i=2}^{n}\sum_{j=1}^{i-1}\mathbb{E}\left[\sigma_{i}^{2}f_{i}^{2}\sigma_{j}^{2}f_{j}^{2}K_{nij}^{2}\psi_{ij}^{2}\right]
 = \dfrac{2n h^{p}}{n-1} \mathbb{E}\left[\sigma_{i}^{2}f_{i}^{2}\sigma_{j}^{2}f_{j}^{2}K_{nij}^{2}\psi_{ij}^{2}\right],
$$
and
{}
\begin{eqnarray*}
\mbox{Var}\left[A_{n}\right] & \leq & \dfrac{64h^{2p}\left\Vert \psi\right\Vert _{\infty}^{4}}{\left(n-1\right)^{4}}\sum_{i=3}^{n}\sum_{j=2}^{i-1}\sum_{j^{\prime}=1}^{j-1}\mathbb{E}\left[\sigma_{i}^{4}f_{i}^{4}\sigma_{j}^{2}f_{j}^{2}\sigma_{j^{\prime}}^{2}f_{j^{\prime}}^{2}K_{nij}^{2}K_{nij^{\prime}}^{2}\right]\\
 &  & +\dfrac{32h^{2p}\left\Vert \psi\right\Vert _{\infty}^{4}}{\left(n-1\right)^{4}}\sum_{i=3}^{n}\sum_{i^{\prime}=1}^{i-1}\sum_{j=2}^{i^{\prime}-1}\mathbb{E}\left[\sigma_{i}^{2}f_{i}^{2}\sigma_{i^{\prime}}^{2}f_{i^{\prime}}^{2}u_{j}^{4}f_{j}^{4}K_{nij}^{2}K_{ni^{\prime}j}^{2}\right]\\
 &  & +\dfrac{16h^{2p}\left\Vert \psi\right\Vert _{\infty}^{4}}{\left(n-1\right)^{4}}\sum_{i=2}^{n}\sum_{j=1}^{i-1}\mathbb{E}\left[\sigma_{i}^{4}f_{i}^{4}u_{j}^{4}f_{j}^{4}K_{nij}^{4}\right]\\
 & = & O\left(n^{-1}\right)\mathbb{E}\left[\sigma_{i}^{4}f_{i}^{4}\sigma_{j}^{2}f_{j}^{2}\sigma_{k}^{2}f_{k}^{2}\mathbf{K}_{nij}\mathbf{K}_{nik}\right]
 +O\left(n^{-1}\right)\mathbb{E}\left[\sigma_{i}^{2}f_{i}^{2}\sigma_{i^{\prime}}^{2}f_{i^{\prime}}^{2}u_{j}^{4}f_{j}^{4}\mathbf{K}_{nij}\mathbf{K}_{ni^{\prime}j}\right] \\ &&+O\left(n^{-2}h^{-p}\right)\mathbb{E}\left[\sigma_{i}^{4}f_{i}^{4}u_{j}^{4}f_{j}^{4}\mathbf{K}_{nij}\right]\\
 & = & O\left(n^{-1}\right)+O\left(n^{-2}h^{-p}\right).
\end{eqnarray*}
Deduce that $\mbox{Var}\left[A_{n}\right]\rightarrow 0,$ and hence remains to show that $\mathbb{E}[A_n]\rightarrow \omega ^2.$ We have
$$
 h^{p}\;\mathbb{E}\left[\sigma_{i}^{2}f_{i}^{2}\sigma_{j}^{2}f_{j}^{2}K_{nij}^{2}\psi_{ij}^{2}\right]\\
 =  \mathbb{E}\left[ \int
 \varphi_{X_{i}}\left(t\right)\varphi_{X_{j}}\left(-t\right)\mathcal{F}\left[K^{2}\right]\left(ht\right)\psi^{2}\left(X_{i}-X_{j}\right)dt\right]
$$
where $\varphi_{x}\left(t\right)=\mathcal{F}\left[\sigma^{2}\left(\cdot,x\right)f^{2}\left(\cdot\right)\pi\left(\cdot\mid x\right)\right]\left(t\right)$.
Let us note that
\begin{eqnarray*}
\mathbb{E}\left[ \int
\left| \varphi_{X_{i}}\left(t\right)\varphi_{X_{j}}\left(-t\right)\right|
 \psi^{2}\left(X_{i}-X_{j}\right)dt\right] &\leq &\|\psi\|_\infty \; \mathbb{E}\left[ \int
\left| \varphi_{X}\left(t\right)\right|^2
 dt\right] \\
  & = & \|\psi\|_\infty \;\mathbb{E}\left[\sigma^{4}\left(W,X\right)f^{4}\left(W\right)\pi\left(W\mid X\right)\right],
\end{eqnarray*}
by Plancherel Theorem. Moreover,
$\mathcal{F}\left[K^{2}\right]\left(ht\right)$ is bounded and
converges pointwise to $\int K^{2}\left(s\right)ds$ as $h\to0$. Then
by Lebesgue's dominated convergence theorem,
\[
 h^{p}\;
  \mathbb{E}\left[\sigma_{i}^{2}f_{i}^{2}\sigma_{j}^{2}f_{j}^{2}K_{nij}^{2}
    \psi_{ij}^{2}\right]
\to \mathbb{E}\left[\int
    \varphi_{X_{i}}\left(t\right)\varphi_{X_{j}}\left(-t\right)
    \psi^{2}\left(X_{i}-X_{j}\right)dt \right] \int
  K^{2}\left(s\right)\, ds =
 \omega^2
\, ,
\]
by Parseval's Theorem.

\medskip
2. By elementary calculations,
\begin{eqnarray*}
\mathbb{E}\left[B_{n}^{2}\right] & = & \dfrac{64h^{2p}}{\left(n-1\right)^{4}}\sum_{i=3}^{n}\sum_{i^{\prime}=3}^{n}\sum_{j=2}^{i-1}
\sum_{j^{\prime}=2}^{i^{\prime}-1}\sum_{k=1}^{j-1}\sum_{k^{\prime}=1}^{j^{\prime}-1}
\mathbb{E}\left[\sigma_{i}^{2}f_{i}^{2}\sigma_{i^{\prime}}^{2}f_{i^{\prime}}^{2}u_{j}
f_{j}u_{j^{\prime}}f_{j^{\prime}}u_{k}f_{k}u_{k^{\prime}}f_{k^{\prime}}\right.\\
 &  & \left.\times K_{nij}K_{ni^{\prime}j^{\prime}}K_{nik}K_{ni^{\prime}k^{\prime}}\psi_{ij}
 \psi_{i^{\prime}j^{\prime}}\psi_{ik}\psi_{i^{\prime}k^{\prime}}\right]\\
 & \leq & \dfrac{64h^{2p}\left\Vert \psi\right\Vert _{\infty}^{4}}{\left(n-1\right)^{4}}\sum_{i=3}^{n}\sum_{i^{\prime}=3}^{n}
 \sum_{j=2}^{i\wedge i^{\prime}-1}\sum_{k=1}^{j-1}\mathbb{E}\left[\sigma_{i}^{2}f_{i}^{2}\sigma_{i^{\prime}}^{2}f_{i^{\prime}}^{2}\sigma_{j}^{2}f_{j}^{2}\sigma_{k}^{2}f_{k}^{2}K_{nij}K_{ni^{\prime}j}K_{nik}K_{ni^{\prime}k}\right]\\
 & = & \dfrac{64h^{2p}\left\Vert \psi\right\Vert _{\infty}^{4}}{\left(n-1\right)^{4}}\sum_{i=3}^{n}\sum_{j=2}^{i-1}\sum_{k=1}^{j-1}\mathbb{E}\left[\sigma_{i}^{4}f_{i}^{4}\sigma_{j}^{2}f_{j}^{2}\sigma_{k}^{2}f_{k}^{2}K_{nij}^{2}K_{nik}^{2}\right]\\
 &  & +\dfrac{128h^{2p}\left\Vert \psi\right\Vert _{\infty}^{4}}{\left(n-1\right)^{4}}\sum_{i=3}^{n}\sum_{i^{\prime}=3}^{i-1}
 \sum_{j=2}^{i^{\prime}-1}\sum_{k=1}^{j-1}\mathbb{E}\left[\sigma_{i}^{2}f_{i}^{2}\sigma_{i^{\prime}}^{2}f_{i^{\prime}}^{2}\sigma_{j}^{2}f_{j}^{2}\sigma_{k}^{2}f_{k}^{2}K_{nij}K_{ni^{\prime}j}K_{nik}K_{ni^{\prime}k}\right]\\
 & = & O\left(n^{-1}\right)\mathbb{E}\left[\sigma_{i}^{4}f_{i}^{4}\sigma_{j}^{2}f_{j}^{2}\sigma_{k}^{2}f_{k}^{2}\mathbf{K}_{nij}\mathbf{K}_{nik}\right] +O\left(h^{p}\right)\mathbb{E}\left[\sigma_{i}^{2}f_{i}^{2}\sigma_{i^{\prime}}^{2}f_{i^{\prime}}^{2}\sigma_{j}^{2}f_{j}^{2}\sigma_{k}^{2}f_{k}^{2}\mathbf{K}_{nij}\mathbf{K}_{ni^{\prime}j}\mathbf{K}_{nik}\right]\\
 & = & O\left(n^{-1}\right)+O\left(h^{p}\right) = o(1)\, .
\end{eqnarray*}


\medskip
3. We have $\forall\varepsilon>0$, $\forall n\geq1$, and $1< i\leq n$,
\begin{eqnarray*}
\mathbb{E}\left[G_{n,i}^{2}I\left(\left|G_{n,i}\right|>\varepsilon\right)
\mid\mathcal{F}_{n,i-1}\right]
  &\leq &
  \mathbb{E}^{1/2}\left[G_{n,i}^{4}\mid\mathcal{F}_{n,i-1}\right]\mathbb{E}^{1/2}
\left[I\left(\left|G_{n,i}\right|>\varepsilon\right)\mid\mathcal{F}_{n,i-1}\right]
\\
& \leq &
\frac{\mathbb{E}\left[G_{n,i}^{4}\mid\mathcal{F}_{n,i-1}\right]}{\varepsilon^{2}}
\, .
\end{eqnarray*}
Then
\begin{eqnarray*}
 \sum_{i=2}^{n}\mathbb{E}\left[G_{n,i}^{2}I\left(\left|G_{n,i}\right|>\varepsilon\right)\mid\mathcal{F}_{n,i-1}\right]
&\leq & \dfrac{1}{\varepsilon^{2}}\sum_{i=2}^{n}\mathbb{E}\left[G_{n,i}^{4}\mid\mathcal{F}_{n,i-1}\right]\\
&\leq & \dfrac{1}{\varepsilon^{2}}\dfrac{16h^{2p}}{\left(n-1\right)^{4}}\sum_{i=2}^{n}\mathbb{E}\left[u_{i}^{4}f_{i}^{4}\mid W_{i},\, X_{i}\right]\left(\sum_{j=1}^{i-1}u_{j}K_{nij}\psi_{ij}\right)^{4}\\
&\leq & \dfrac{1}{\varepsilon^{2}}\dfrac{16\kappa_{4}h^{2p}}{\left(n-1\right)^{4}}\sum_{i=2}^{n}\left(\sum_{j=1}^{i-1}u_{j}K_{nij}\psi_{ij}\right)^{4},
\end{eqnarray*}
where $\kappa_{4}$ is any constant that bounds $\mathbb{E}\left[u^{4}f^{4}\mid W,\, X\right].$ The last expression that multiplies $\varepsilon^{-2}$ is positive and has expectation
\begin{align*}
 & \!\!\!\!\! \dfrac{16\kappa_{4}h^{2p}}{\left(n-1\right)^{4}}\sum_{i=2}^{n}\sum_{j_{1}=1}^{i-1}\sum_{j_{2}=1}^{i-1}\sum_{j_{3}=1}^{i-1}\sum_{j_{4}=1}^{i-1}\mathbb{E}\left[u_{j_{1}}f_{j_{1}}u_{j_{2}}f_{j_{2}}u_{j_{3}}j_{j_{3}}u_{j_{4}}f_{j_{4}}\vphantom{K_{nij_{1}}K_{nij_{2}}K_{nij_{3}}K_{nij_{4}}\psi_{ij_{1}}\psi_{ij_{2}}\psi_{ij_{3}}\psi_{ij_{4}}}\right.\\
 & \hphantom{\dfrac{16\kappa_{4}h^{2p}}{\left(n-1\right)^{4}}\sum_{i=2}^{n}\sum_{j_{1}=1}^{i-1}\sum_{j_{2}=1}^{i-1}\sum_{j_{3}=1}^{i-1}\sum_{j_{4}=1}^{i-1}\mathbb{E}\left[\right.}\left.\times\vphantom{u_{j_{1}}f_{j_{1}}u_{j_{2}}f_{j_{2}}u_{j_{3}}j_{j_{3}}u_{j_{4}}f_{j_{4}}}K_{nij_{1}}K_{nij_{2}}K_{nij_{3}}K_{nij_{4}}\psi_{ij_{1}}\psi_{ij_{2}}\psi_{ij_{3}}\psi_{ij_{4}}\right]\\
= & \;\;\dfrac{96\kappa_{4}h^{2p}}{\left(n-1\right)^{4}}\sum_{i=3}^{n}\sum_{j=1}^{i-1}\sum_{k=1}^{j-1}\mathbb{E}\left[u_{j}^{2}f_{j}^{2}u_{k}^{2}f_{k}^{2}K_{nij}^{2}K_{nik}^{2}\psi_{ij}^{2}\psi_{ik}^{2}\right]\\
 & +\dfrac{16\kappa_{4}h^{2p}}{\left(n-1\right)^{4}}\sum_{i=2}^{n}\sum_{j=1}^{i-1}\mathbb{E}\left[u_{j}^{4}f_{j}^{4}K_{nij}^{4}\psi_{ij}^{4}\right]\\
= & \;\;O\left(n^{-1}\right)\mathbb{E}\left[u_{j}^{2}f_{j}^{2}u_{k}^{2}f_{k}^{2}\mathbf{K}_{nij}\mathbf{K}_{nik}\right]+O\left(n^{-2}h^{-p}\right)\mathbb{E}\left[u_{j}^{4}f_{j}^{4}\mathbf{K}_{nij}\right]\\
= & \;\;O\left(n^{-1}\right)+O\left(n^{-2}h^{-p}\right).
\end{align*}
The desired result follows.
\end{proof}


The following result, known as Bochner's Lemma (see Theorem 1.1.1. of
\cite{Bochner1955}) will be repeatedly use in the following. We recall
it for the sake of completeness.
\begin{lem}
\label{Bochner}For any function $l\left(\cdot\right)\in{\cal U}^{p}$
and any integrable kernel $K\left(\cdot\right)$,
\[
\sup_{x\in\mathbb{R}^{p}}\left|\int l\left(y\right)\frac{1}{h^{p}}K\left(\frac{x-y}{h}\right)\, dy-l\left(x\right)\int K\left(u\right)\, du\right|\rightarrow0.
\]
\end{lem}

\medskip

In the following we provide the proofs for rates for the remaining
terms in the decomposition of $I_n$, see Propositions \ref{Ustat} and
\ref{Remaining}.  For this purpose, we use the following a
decomposition for $U-$statistics that can be found in
\cite{Lavergne2000}: if
$U_{n}=\left(1/n^{\left(m\right)}\right)\sum_{a}H_{n}\left(Z_{i_{1}},\,\dots,\,
Z_{i_{m}}\right)$, then
$$
\mathbb{E}\left[U_{n}^{2}\right] = \left(\dfrac{1}{n^{\left(m\right)}}\right)^{2}\sum_{c=0}^{m}\dfrac{n^{\left(2m-c\right)}}{c!}\sum_{\left|\Delta_{1}\right|=c=\left|\Delta_{2}\right|}^{\left(c\right)}I\left(\Delta_{1},\Delta_{2}\right)= \sum_{c=0}^{m}O\left(n^{-c}\right)\sum_{\left|\Delta_{1}\right|=c=\left|\Delta_{2}\right|}^{\left(c\right)}I\left(\Delta_{1},\Delta_{2}\right),
$$
where $\sum^{\left(c\right)}$ denotes summation over sets $\Delta_{1}$
and $\Delta_{1}$ of ordered positions of length $c$,
\[
I\left(\Delta_{1},\Delta_{2}\right)=\mathbb{E}\left[H_{n}\left(Z_{i_{1}},\,\dots,\, Z_{i_{m}}\right)H_{n}\left(Z_{j_{1}},\,\dots,\, Z_{j_{m}}\right)\right]
\]
and the $i$'s position in $\Delta_{1}$ coincide with the $j$'s
position in $\Delta_{2}$ and are pairwise distinct otherwise. Now,
we will bound $\mathbb{E}\left[U_{n}^{2}\right]$ using the $\xi_{c}=\sum^{\left(c\right)}I\left(\Delta_{1},\Delta_{2}\right)$
and the fact that by Cauchy's inequality,
\begin{eqnarray*}
I^{2}\left(\Delta_{1},\Delta_{2}\right) & = & \mathbb{E}^{2}\left[\mathbb{E}\left[H_{n}\left(Z_{i_{1}},\,\dots,\, Z_{i_{m}}\right)\mid Z_{c}\right]\mathbb{E}\left[H_{n}\left(Z_{j_{1}},\,\dots,\, Z_{j_{m}}\right)\mid Z_{c}\right]\right]\\
 & \leq & \mathbb{E}\left[\mathbb{E}^{2}\left[H_{n}\left(Z_{i_{1}},\,\dots,\, Z_{i_{m}}\right)\mid Z_{c}\right]\right]\mathbb{E}\left[\mathbb{E}^{2}\left[H_{n}\left(Z_{j_{1}},\,\dots,\, Z_{j_{m}}\right)\mid Z_{c}\right]\right]
\end{eqnarray*}
where $Z_{c}$ denotes the common $Z_{i}$'s.

\begin{proof}[Proof of Proposition \ref{Ustat}]
After bounding the $\psi_{ij}$'s by $\left\Vert \psi\right\Vert _{\infty}$
the arguments are very similar to those used in \cite{Lavergne2000}.
We prove only the first statement.
\begin{description}
\item [{(i)}] $I_{1,3}$ is a U-statistic with kernel $H_{n}\left(Z_{i},Z_{j},Z_{l}\right)=u_{i}f_{i}u_{l}L_{njl}K_{nij}\psi_{ij}.$
We need to bound the $\xi_{c}$, $c=0,1,2,3$.

\begin{enumerate}
\item $\mathbb{E}\left[H_{n}\right]=0,$ thus $\xi_{0}=0$.
\item $\xi_{1}=O\left(\delta_{n}^{2}\right)$. Indeed,
$\mathbb{E}\left[H_{n}\mid Z_{l}\right]=\delta_{n}u_{l}\mathbb{E}\left[d_{i}f_{i}L_{njl}K_{nij}\psi_{ij}\mid Z_{l}\right]$ and $\mathbb{E}\left[H_{n}\mid Z_{i}\right]=0=\mathbb{E}\left[H_{n}\mid Z_{j}\right].$
Then
\begin{eqnarray*}
\mathbb{E}\left[\mathbb{E}^{2}\left[H_{n}\mid Z_{l}\right]\right] & \leq & \left\Vert \psi\right\Vert _{\infty}^{2}\delta_{n}^{2}\mathbb{E}\left[u_{l}^{2}\mathbb{E}^{2}\left[d_{i}f_{i}L_{njl}K_{nij}\mid Z_{l}\right]\right]\\
 & = & O\left(\delta_{n}^{2}\right)\mathbb{E}\left[u_{l}^{2}\mathbb{E}^{2}\left[L_{njl}d_{j}f_{j}^{2}\mid Z_{l}\right]\right]=O\left(\delta_{n}^{2}\right).
\end{eqnarray*}

\item $\xi_{2}=O\left(g^{-p}\right)$. Indeed, we have
\begin{eqnarray*}
\mathbb{E}\left[H_{n}\mid Z_{i},Z_{j}\right] & = & u_{i}f_{i}K_{nij}\psi_{ij}\mathbb{E}\left[u_{l}L_{njl}\mid Z_{j}\right]=0,\\
\mathbb{E}\left[H_{n}\mid Z_{i},Z_{l}\right] & = & u_{i}f_{i}u_{l}\mathbb{E}\left[L_{njl}K_{nij}\psi_{ij}\mid Z_{i},Z_{l}\right],\\
\mathbb{E}\left[H_{n}\mid Z_{j},Z_{l}\right] & = & u_{l}L_{njl}\mathbb{E}\left[u_{i}f_{i}K_{nij}\psi_{ij}\mid Z_{j}\right]=\delta_{n}u_{l}L_{njl}\mathbb{E}\left[d_{i}f_{i}K_{nij}\psi_{ij}\mid Z_{j}\right].
\end{eqnarray*}
By successive applications of Lemma \ref{Bochner},
\begin{eqnarray*}
\mathbb{E}\left[\mathbb{E}^{2}\left(H_{n}\mid Z_{i},Z_{l}\right)\right] & \leq & \left\Vert \psi\right\Vert _{\infty}^{2}\mathbb{E}\left[u_{i}^{2}f_{i}^{2}u_{l}^{2}\mathbb{E}\left[L_{njl}K_{nij}\mid Z_{i},Z_{l}\right]\mathbb{E}\left[L_{nj^{\prime}l}K_{nij^{\prime}}\mid Z_{i},Z_{l}\right]\right]\\
 & = & O\left(g^{-p}\right)\mathbb{E}\left[u_{i}^{2}f_{i}^{2}u_{l}^{2}\mathbb{E}\left[\mathbf{L}_{njl}\mathbf{K}_{nij}\mid Z_{i},Z_{l}\right]\mathbb{E}\left[\mathbf{K}_{nij^{\prime}}\mid Z_{i},Z_{l}\right]\right]\\
 & = & O\left(g^{-p}\right)\mathbb{E}\left[u_{i}^{2}f_{i}^{3}u_{l}^{2}\mathbf{L}_{njl}\mathbf{K}_{nij}\right]=O\left(g^{-p}\right),\\
\mathbb{E}\left[\mathbb{E}^{2}\left[H_{n}\mid Z_{j},Z_{l}\right]\right] & \leq & \left\Vert \psi\right\Vert _{\infty}^{2}\delta_{n}^{2}\mathbb{E}\left[u_{l}^{2}L_{njl}^{2}\mathbb{E}^{2}\left[d_{i}f_{i}K_{nij}\mid Z_{j}\right]\right]\\
 & \leq & O\left(\delta_{n}^{2}\right)\mathbb{E}\left[u_{l}^{2}L_{njl}^{2}d_{j}^{2}f_{j}^{4}\right]\\
 & = & O\left(\delta_{n}^{2}\right) O\left(g^{-p}\right)\mathbb{E}\left[u_{l}^{2}\mathbf{L}_{njl}d_{j}^{2}f_{j}^{4}\right]
 =O\left(g^{-p}\right).
\end{eqnarray*}

\item $\xi_{3}=O\left(g^{-p}h^{-p}\right)$, as $\mathbb{E}\left[H_{n}^{2}\right]$
equals
\[
\mathbb{E}\left[u_{i}^{2}u_{l}^{2}f_{i}^{2}L_{njl}^{2}K_{nij}^{2}\psi_{ij}^{2}\right]=O\left(g^{-p}h^{-p}\right)\mathbb{E}\left[u_{i}^{2}u_{l}^{2}f_{i}^{2}\mathbf{L}_{njl}\mathbf{K}_{nij}\right]=O\left(g^{-p}h^{-p}\right).
\]

\end{enumerate}
\end{description}
Collecting results, $\mathbb{E}\left[\left(nh^{p/2}I_{1,3}\right)^{2}\right]=
O\left(\delta_{n}^{2}nh^{p}\right)+
O\left(h^{p}/g^{p}\right)+
O\left(n^{-1}g^{-p}\right)=o(1)$.
\end{proof}

\begin{proof}[Proof of Proposition \ref{Remaining}]
As in Proposition \ref{Ustat}, we only prove the first statement.
We will use the following lemma, which is similar to
Lemma 2 of \cite{Lavergne2000}, and whose proof is then omitted.
\begin{lem}
\label{Lambda}
Let  $\Delta f_{i}^{j}= \widehat{f}_{i}^{j}-f_{i}.$
If  $f\left(\cdot\right)\in{\cal U}^{p}$ and $ng^{p}\rightarrow\infty$,
$\mathbb{E}\left[\Delta^{2}f_{i}^{j}\mid Z_{i},Z_{j},Z_{i'},Z_{j'}\right]=o\left(1\right)$
and $E\left[\Delta^{2}f_{i}^{j,l}\mid Z_{i},Z_{j},Z_{l},Z_{i'},Z_{j'},Z_{l'}\right]=o\left(1\right)$
uniformly in the indices.
\end{lem}
\begin{description}
\item [{(i)}] Let us  denote $\Delta f_{i}^{j}= \widehat{f}_{i}^{j}-f_{i}.$
We have $I_{1,1}=\left(1/n^{\left(2\right)}\right)\sum_{a}u_{i}\Delta f_{i}^{j}u_{j}f_{j}K_{nij}\psi_{ij}$
so that
\begin{equation}\label{biz_biz}
\mathbb{E}\left[I_{1,1}^{2}\right]=\left(\frac{1}{n^{\left(2\right)}}\right)^{2}\left[\sum_{a}u_{i}\Delta f_{i}^{j}u_{j}f_{j}K_{nij}\psi_{ij}\right]\left[\sum_{a}u_{i^{\prime}}\Delta f_{i^{\prime}}^{j^{\prime}}u_{j^{\prime}}f_{j^{\prime}}K_{ni^{\prime}j^{\prime}}\psi_{i^{\prime}j^{\prime}}\right],
\end{equation}
where the first (respectively the second) sum is taken over all
arrangements of different indices $i$ and $j$ (respectively different
indices $i^{\prime}$ and $j^{\prime}$).  Let $\overline{W}$ denote the
sample of $W_{i},$ $1\leq i \leq n,$ and let
$\lambda_{n}=E\left[\Delta^{2}f_{i}^{j}\mid
  Z_{i},Z_{j},Z_{i^{\prime}},Z_{j^{\prime}}\right]$. By Lemma
\ref{Lambda}, $\lambda_n=o\left(1\right)$ uniformly in the indices. By
Equation (\ref{biz_biz}), $\mathbb{E}\left[I_{1,1}^{2}\right]$ is
equal to a normalized sum over four indices. This sum could split in
three sums of the following types.
\begin{enumerate}
\item All indices are different, that is a sum of $n^{\left(4\right)}$ terms. Each term in the sum can be bounded as follows:
\[
\begin{array}{cl}
 & \mathbb{E}\left[u_{i}\Delta f_{i}^{j}u_{j}f_{j}K_{nij}\psi_{ij}u_{i^{\prime}}\Delta f_{i^{\prime}}^{j^{\prime}}u_{j^{\prime}}f_{j^{\prime}}K_{ni^{\prime}j^{\prime}}\psi_{i^{\prime}j^{\prime}}\right]\\
\leq & \left\Vert \psi\right\Vert _{\infty}^{2}\delta_{n}^{4}\mathbb{E}\left[\Delta f_{i}^{j}f_{j}\Delta f_{i^{\prime}}^{j^{\prime}}f_{j^{\prime}}\mathbb{E}\left[d_{i}d_{j}d_{i'}d_{j^{\prime}}K_{nij}K_{ni^{\prime}j^{\prime}}\mid\overline{W}\right]\right]\\
\leq & \left\Vert \psi\right\Vert _{\infty}^{2}\delta_{n}^{4}\mathbb{E}\left[f_{j}f_{j^{\prime}}d_{i}d_{j}d_{i^{\prime}}d_{j^{\prime}}K_{nij}K_{ni^{\prime}j^{\prime}}\mathbb{E}\left[\Delta f_{i}^{j}\Delta f_{i^{\prime}}^{j^{\prime}}\mid Z_{i},Z_{j},Z_{i^{\prime}},Z_{j^{\prime}}\right]\right]\\
\leq & O(\delta_{n}^{4}\lambda_{n})\mathbb{E}\left|f_{j}f_{j^{\prime}}d_{i}d_{j}d_{i^{\prime}}d_{j^{\prime}}K_{nij}K_{ni^{\prime}j^{\prime}}\right|=O\left(\delta_{n}^{4}\lambda_{n}\right).
\end{array}
\]

\item One index is common to $\left\{ i,j\right\} $ and $\left\{ i^{\prime},j^{\prime}\right\} ,$
that is a sum of $4n^{\left(3\right)}$ terms.  For each of such terms we can write
\[
\begin{array}{ccl}
\left(i^{\prime}=i\right)\quad &  & \mathbb{E}\left[u_{i}^{2}\Delta f_{i}^{j}u_{j}f_{j}K_{nij}\psi_{ij}\Delta f_{i}^{j^{\prime}}u_{j^{\prime}}f_{j^{\prime}}K_{nij^{\prime}}\psi_{ij^{\prime}}\right]\\
 & \leq & \left\Vert \psi\right\Vert _{\infty}^{2}\delta_{n}^{2}\mathbb{E}\left[\Delta f_{i}^{j}f_{j}\Delta f_{i}^{j^{\prime}}f_{j^{\prime}}E\left[u_{i}^{2}d_{j}d_{j^{\prime}}K_{nij}K_{nij^{\prime}}\mid\overline{W}\right]\right]\\
 & \leq & O(\delta_{n}^{2}\lambda_{n})\mathbb{E}\left|f_{j}f_{j^{\prime}}u_{i}^{2}d_{j}d_{j^{\prime}}K_{nij}K_{nij^{\prime}}\right|=O\left(\delta_{n}^{2}\lambda_{n}\right),\\
\\
\left(j^{\prime}=j\right)\quad &  & \mathbb{E}\left[u_{i}\Delta f_{i}^{j}u_{j}^{2}f_{j}^{2}K_{nij}\psi_{ij}u_{i^{\prime}}\Delta f_{i^{\prime}}^{j}K_{ni^{\prime}j}\psi_{i^{\prime}j}\right]\\
 & \leq & \left\Vert \psi\right\Vert _{\infty}^{2}\delta_{n}^{2}\mathbb{E}\left[\Delta f_{i}^{j}f_{j}^{2}\Delta f_{i^{\prime}}^{j}\mathbb{E}\left[d_{i}u_{j}^{2}d_{i^{\prime}}K_{nij}K_{ni^{\prime}j}\mid\overline{W}\right]\right]\\
 & \leq & O(\delta_{n}^{2}\lambda_{n})\mathbb{E}\left|f_{j}^{2}d_{i}u_{j}^{2}d_{i'}K_{nij}K_{ni^{\prime}j}\right|=O\left(\delta_{n}^{2}\lambda_{n}\right),\\
\\
\left(i^{\prime}=j\right)\quad &  & \mathbb{E}\left[u_{i}\Delta f_{i}^{j}u_{j}^{2}f_{j}K_{nij}\psi_{ij}\Delta f_{j}^{j^{\prime}}u_{j^{\prime}}f_{j^{\prime}}K_{njj^{\prime}}\psi_{jj^{\prime}}\right]\\
 & \leq & \left\Vert \psi\right\Vert _{\infty}^{2}\delta_{n}^{2}\mathbb{E}\left[\Delta f_{i}^{j}f_{j}\Delta f_{j}^{j^{\prime}}f_{j^{\prime}}E\left[d_{i}u_{j}^{2}d_{j'}K_{nij}K_{njj^{\prime}}\mid\overline{W}\right]\right]\\
 & \leq & O(\delta_{n}^{2}\lambda_{n})\mathbb{E}\left|f_{j}f_{j^{\prime}}d_{i}u_{j}^{2}d_{j^{\prime}}K_{nij}K_{njj^{\prime}}\right|=O\left(\delta_{n}^{2}\lambda_{n}\right).
\end{array}
\]
The case $j^{\prime}=i$ is similar to $i^{\prime}=j$.
\item Two indices in common to $\left\{ i,j\right\} $ and $\left\{ i^{\prime},j^{\prime}\right\}, $
that is a sum of $2n^{\left(2\right)}$ terms. For each term in the sum we can write
\[
\mathbb{E}\left[u_{i}^{2}u_{j}^{2}\left(\Delta f_{i}^{j}\right)^{2}\!f_{j}^{2}K_{nij}^{2}\psi_{ij}^{2}\right]\!=O\!\left(\lambda_{n}h^{-p}\right)\;\mbox{ and }\;\mathbb{E}\left[u_{i}^{2}u_{j}^{2}\Delta f_{i}^{j}\Delta f_{j}^{i}f_{i}f_{j}K_{nij}^{2}\psi_{ij}^{2}\right]\!=O\!\left(\lambda_{n}h^{-p}\right).
\]
\end{enumerate}
\end{description}
Therefore, $
\mathbb{E}\left[\left(nh^{p/2}I_{1,1}\right)^{2}\right]=\delta_{n}^{4}n^{2}h^{p}O\left(\lambda_{n}\right)+\delta_{n}^{2}nh^{p}O\left(\lambda_{n}\right)+O\left(\lambda_{n}\right)
=O\left(\lambda_{n}\right) $.  The result then follows from Lemma
\ref{Lambda}.
\end{proof}


\begin{proof}[Proof of Lemma \ref{unif_omeg}]
We only prove the result for $\Delta \hat r_i \hat f_i, $ as the
reasoning is similar for $\Delta\hat f_i$. We have
\begin{eqnarray*}
\Delta \hat r_i \hat f_i &=& \frac{1}{(n-1)g^p} \sum_{k\neq i}
\left\{Y_k L\left( (W_i-W_k)g^{-1} \right) -\mathbb{E}\left[Y L\left(
  (W_i - W) g^{-1} \right)\right]\right\}\\ && + \mathbb{E}\left[r(W)
  g^{-p} L\left( (W_i-W) g^{-1} \right)\right] - r(W_i) f(W_i)\\ &=&
\Delta_{1i} + \Delta_{2i}.
\end{eqnarray*}
The uniform continuity of $r(\cdot)f(\cdot)$ implies
$\sup_i|\Delta_{2i}|=o_p(1)$ by Lemma \ref{Bochner}.  {For
  $\sup_i|\Delta_{1i}|$, we use empirical process tools. Let us
  introduce some notation. Let $\mathcal{G}$ be a class of functions
  of the observations with envelope function $G$ and let
$$ J(\delta,\mathcal{G}, L^2 )=\sup_Q \int_0^\delta \sqrt{1+\ln N
  (\varepsilon \|G\|_{2},\mathcal{G}, L^2(Q) ) } d\varepsilon
,\qquad 0<\delta\leq 1,
$$ denote the uniform entropy integral, where the supremum is taken
  over all finitely discrete probability distributions $Q$ on the
  space of the observations, and $\| G \|_{2}$ denotes the norm of
  $G$ in $L^2(Q)$. Let $Z_1,\cdots,Z_n$ be a sample of independent
  observations and let
\begin{equation*}
\mathbb{G}_n g=\frac{1}{\sqrt{n}}\sum_{i=1}^n \gamma(Z_i) , \qquad \gamma \in\mathcal{G}
\end{equation*}
be the empirical process indexed by $\mathcal{G}$. If the covering
number $N (\varepsilon ,\mathcal{G}, L^2(Q) ) $ is of polynomial order
in $1/\varepsilon,$ there exists a constant $c>0$ such that
$J(\delta,\mathcal{G}, L^2 )\leq c \delta \sqrt{\ln(1/\delta)}$ for
$0<\delta<1/2.$ Now if $\mathbb{E}\gamma^2 < \delta^2 \mathbb{E}G^2$
for every $\gamma$ and some $0<\delta <1$, and
$\mathbb{E}G^{(4\upsilon-2)/(\upsilon-1)}<\infty$ for some
$\upsilon>1$, under mild additional measurability conditions, Theorem
3.1 of \cite{Vaart2011} implies
\begin{equation}\label{vwww0}
 \sup_{\mathcal{G}}|\mathbb{G}_n \gamma| = J(\delta,\mathcal{G}, L^2
 )\left( 1 + \frac{ J(\delta^{1/\upsilon},\mathcal{G}, L^2 )}{\delta^2
   \sqrt{n} }
 \frac{\|G\|_{(4\upsilon-2)/(\upsilon-1)}^{2-1/\upsilon}}{\|G\|_{2}^{2-1/\upsilon}}
 \right)^{\upsilon/(2\upsilon-1)} \|G\|_2 O_p(1),
\end{equation}
where $\|G\|_{2}^2 = \mathbb{E}G^2$ and the $ O_p(1)$ term is
independent of $n.$ Note that the family $\mathcal{G}$ could change
with $n$, as soon as the envelope is the same for all $n$.  {We apply
  this result to the family of functions $\mathcal{G} = \{ Y L ((W -
  w)/g) : w\in\mathbb{R}^p\}$ for a sequence $g$ that converges to
  zero and the envelope $G(Y,W)=Y\sup_{w\in\mathbb{R}^p} L(w).$ Its
  entropy number is of polynomial order in $1/\varepsilon$,
  independently of $n$, as $L(\cdot)$ is of bounded variation, see for
  instance \cite{Vaart1996}.  Now for any $\gamma \in \mathcal{G}$, $
  \mathbb{E} \gamma ^2(Y,W) \leq C g^p \mathbb{E} G^2(Y,W), $ for some
  constant $C$.  Let $\delta = g^{3p/7},$ so that $ \mathbb{E} \gamma
  ^2(Y,W) \leq C^\prime \delta^2 \mathbb{E} G^2(Y,W), $ for some
  constant $C ^\prime$ and $\upsilon =3/2$, which corresponds to
  $\mathbb{E}G^{8}<\infty$ that is guaranteed by our assumptions.  The
  bound in (\ref{vwww0}) thus yields
$$
 \sup_{\mathcal{G}}\left|\frac{1}{g^p \sqrt{n}} \; \mathbb{G}_n \gamma\right| =  \frac{ \ln^{1/2}(n)}{g^{4p/7} \sqrt{n}}
 \left[ 1 + n^{-1/2}g ^{-4p/7}\ln^{1/2}(n) \right]^{3/4} O_p(1) ,
$$
where the $ O_p(1)$ term is independent of $n$.
Since $n^{7/8} g^p/\ln n  \rightarrow
\infty,$ the expected result follows.
}
}
\end{proof}


\begin{proof}[Proof of Lemma \ref{DeltaUiStar}]
We have
\begin{eqnarray*}
\hat{u}_{i}^{*}\hat{f}_{i} & = & \dfrac{1}{n-1}\sum_{k\neq
  i}\left(Y_{i}^{*}-Y_{k}^{*}\right)L_{nik}\\ & = &
u_{i}^{*}\hat{f}_{i}-\dfrac{1}{n-1}\sum_{k\neq
  i}u_{k}^{*}L_{nik}+\dfrac{1}{n-1}\sum_{k\neq
  i}\left(\hat{r}_{i}-\hat{r}_{k}\right)L_{nik}
\end{eqnarray*}
where
\begin{eqnarray*}
\dfrac{1}{n-1}\sum_{k\neq i}\left(\hat{r}_{i}-\hat{r}_{k}\right)L_{nik} & = & \dfrac{1}{n-1}\sum_{k\neq i}\left(r_{i}-r_{k}\right)L_{nik}
 +\left(\hat{r}_{i}-r_{i}\right)\hat{f}_{i}\\
 &  & -\dfrac{1}{\left(n-1\right)^{2}\hat{f}_{k}}\sum_{k\neq i}\sum_{k^{\prime}\neq k}\left(r_{k^{\prime}}-r_{k}\right)L_{nkk^{\prime}}L_{nik}\\
 &  & -\dfrac{1}{\left(n-1\right)^{2}\hat{f}_{k}}\sum_{k\neq i}\sum_{k^{\prime}\neq k}u_{k^{\prime}}L_{nkk^{\prime}}L_{nik}.
\end{eqnarray*}
By Lemma \ref{unif_omeg} and the fact that $f(\cdot)$ is bounded away from zero, deduce that $\sup_i |\hat{r}_{i}-r_{i}|  = o_{p}\left(1\right).$ From this and applying several times the arguments in the proof of Lemma \ref{unif_omeg} we obtain
$$
\dfrac{1}{n-1}\sum_{k\neq i}\left(\hat{r}_{i}-\hat{r}_{k}\right)L_{nik}  =  o_{p}\left(1\right).
$$
On the other hand,
\begin{eqnarray*}
\left|\dfrac{1}{n-1}\sum_{k\neq i}u_{k}^{*}L_{nik}\right| & \leq  &\left|\dfrac{1}{n-1}\sum_{k\neq i}\eta_{k} u_k L_{nik}\right| +\dfrac{\sup_j |\hat{r}_{j}-r_{j}|}{n-1}\sum_{k\neq i}|\eta_k| \bf{L}_{nik}\\
&=& o_{p}\left(1\right),
\end{eqnarray*}
where we used again the arguments for $\Delta_{1i}$ in the proof of Lemma \ref{unif_omeg} (here with $\eta_{k} u_k$ and $|\eta_k|$ in the place of $Y_k$) to derive the last rate.
\end{proof}

\newpage

\begin{figure}
\begin{centering}
\begin{tabular}{c}
 \includegraphics[scale=0.6,angle=270]{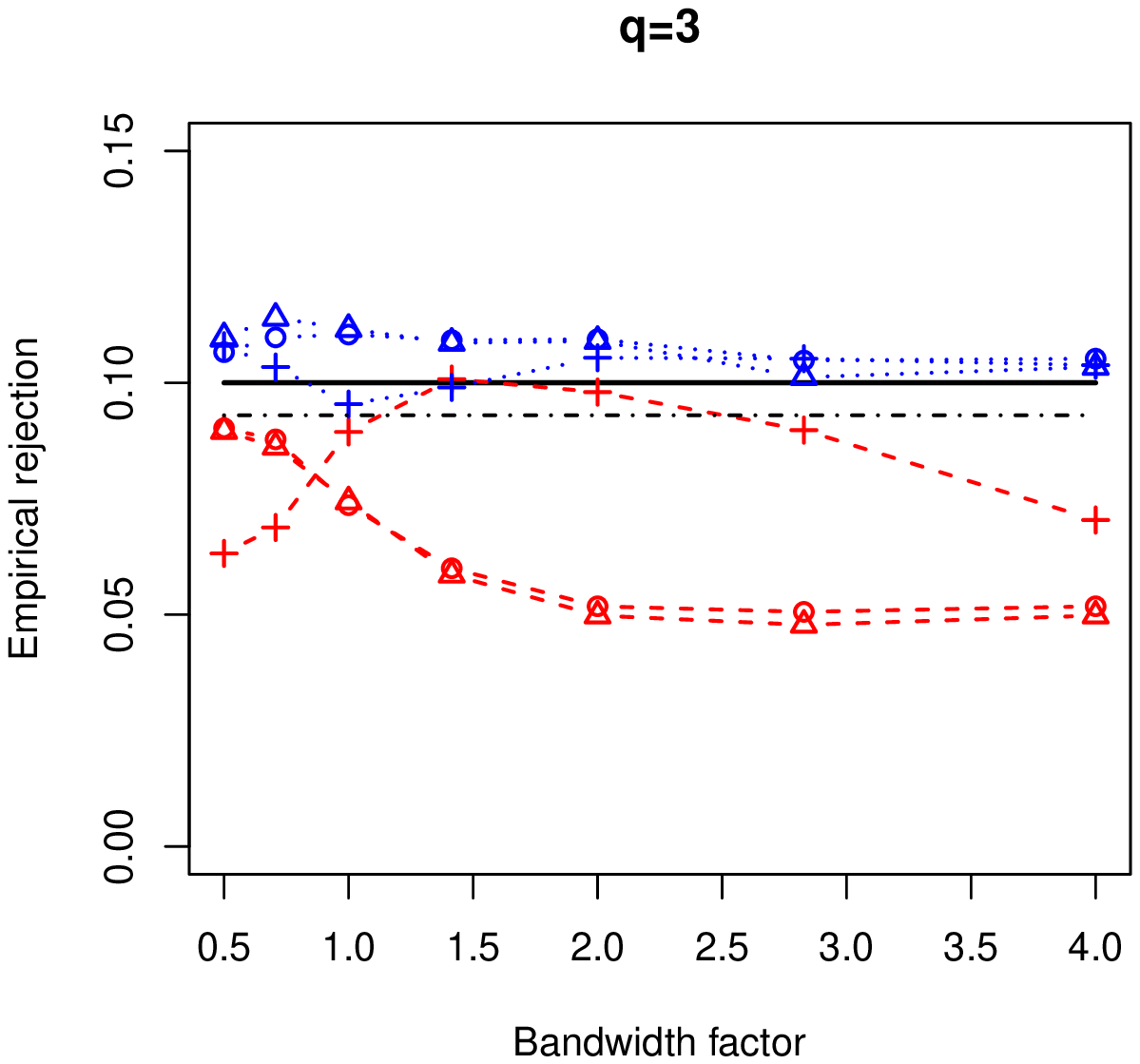} ~\!\!
 \includegraphics[scale=0.6,angle=270]{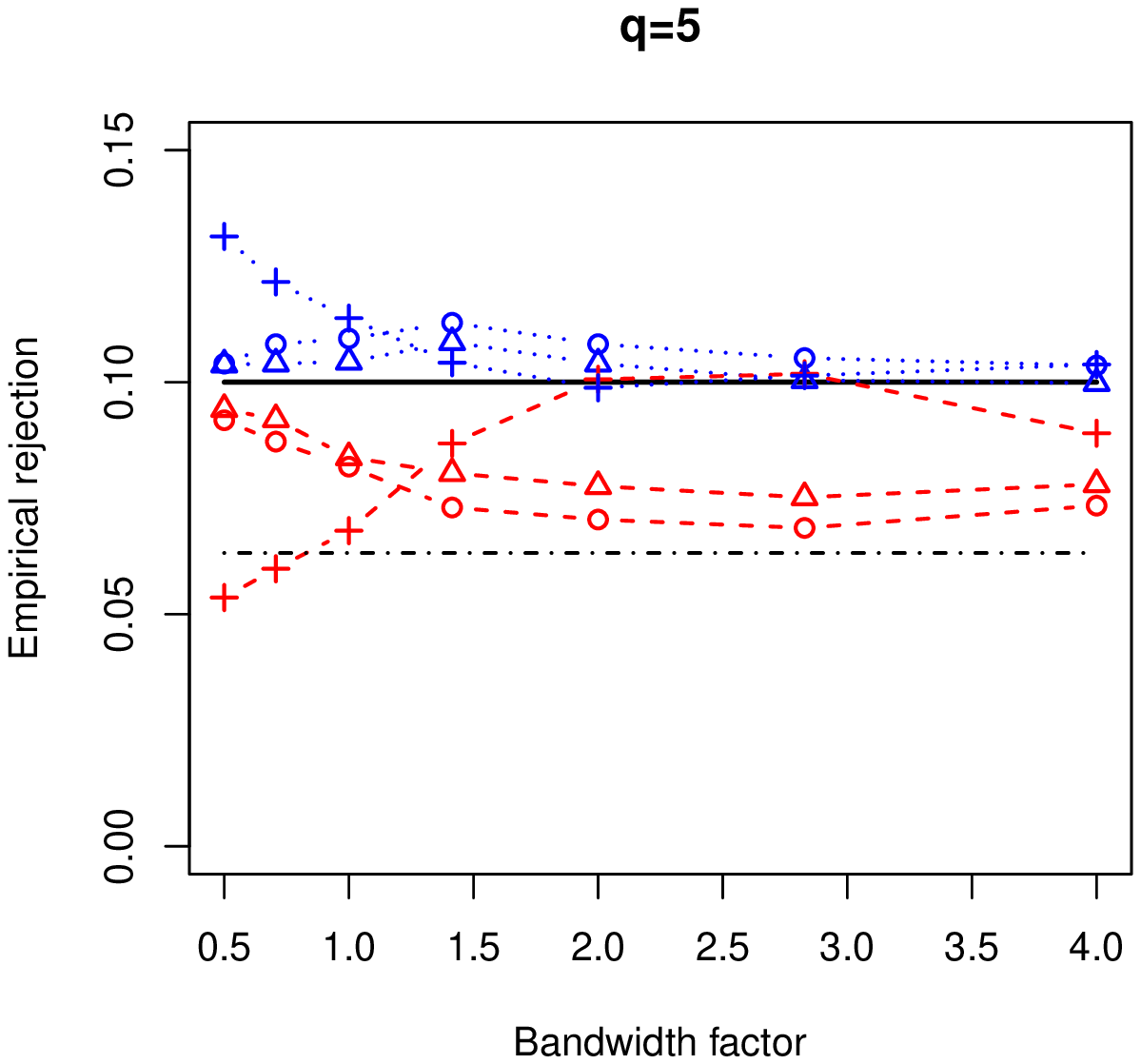}\tabularnewline
 \includegraphics[scale=0.6,angle=270]{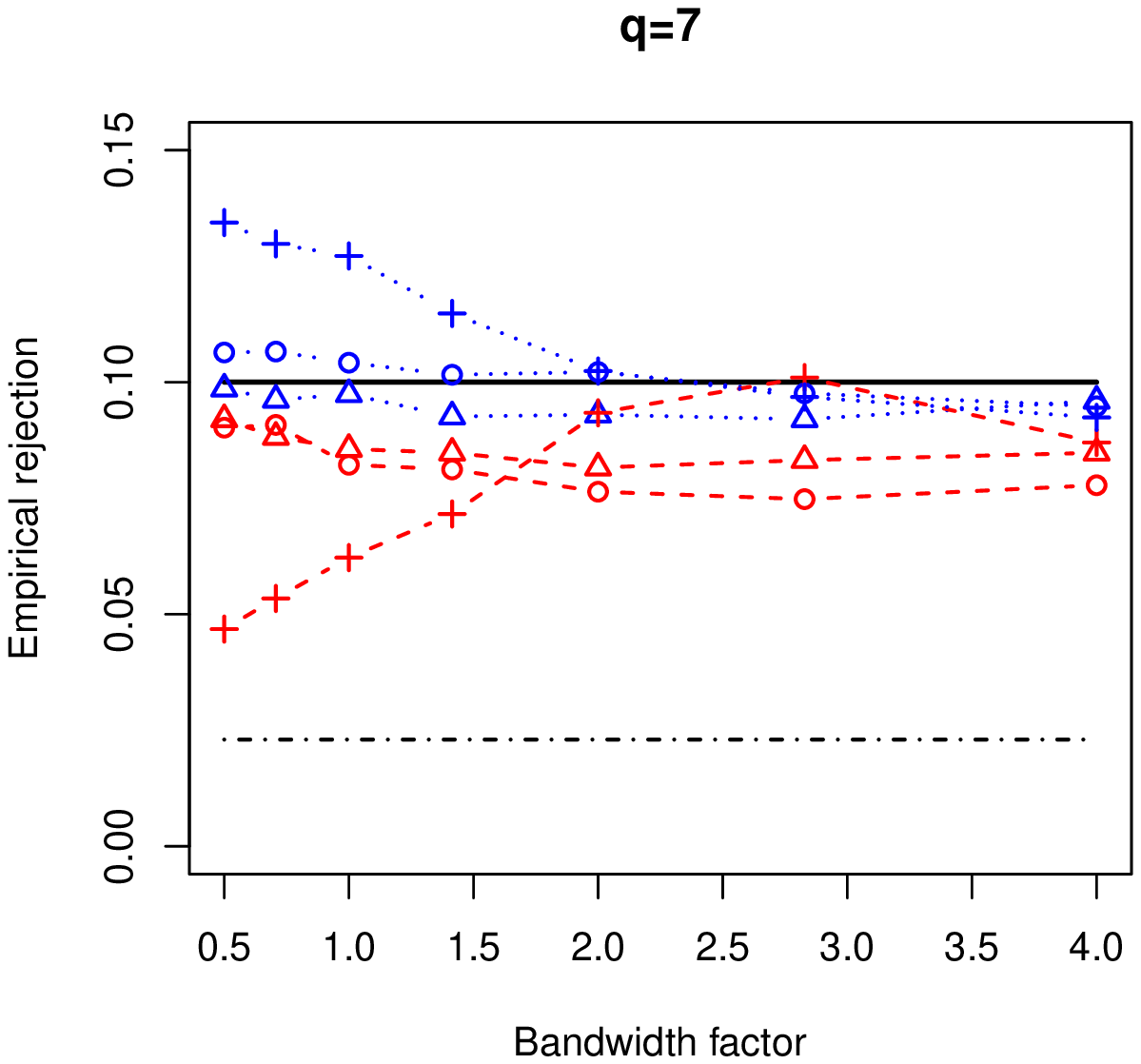} ~\!\!
 \includegraphics[scale=0.6,angle=270]{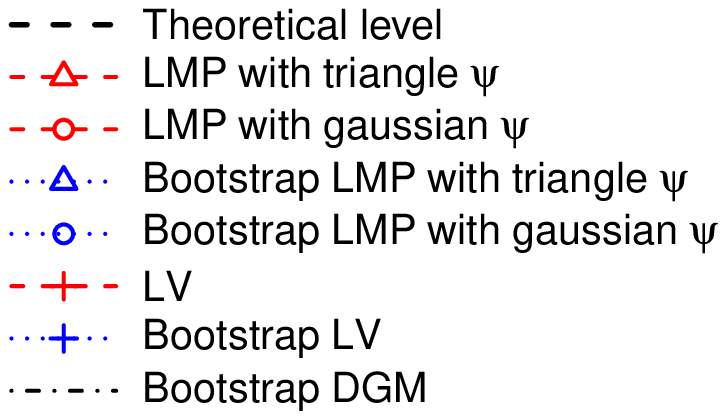}\tabularnewline
\end{tabular}
\par\end{centering}
\caption{Empirical rejections under $H_0$ as a function of the bandwidth, $n=100$  \label{fig:LevelCont}}
\end{figure}

\begin{figure}
\begin{centering}
\begin{tabular}{c}
 \includegraphics[scale=0.6,angle=270] {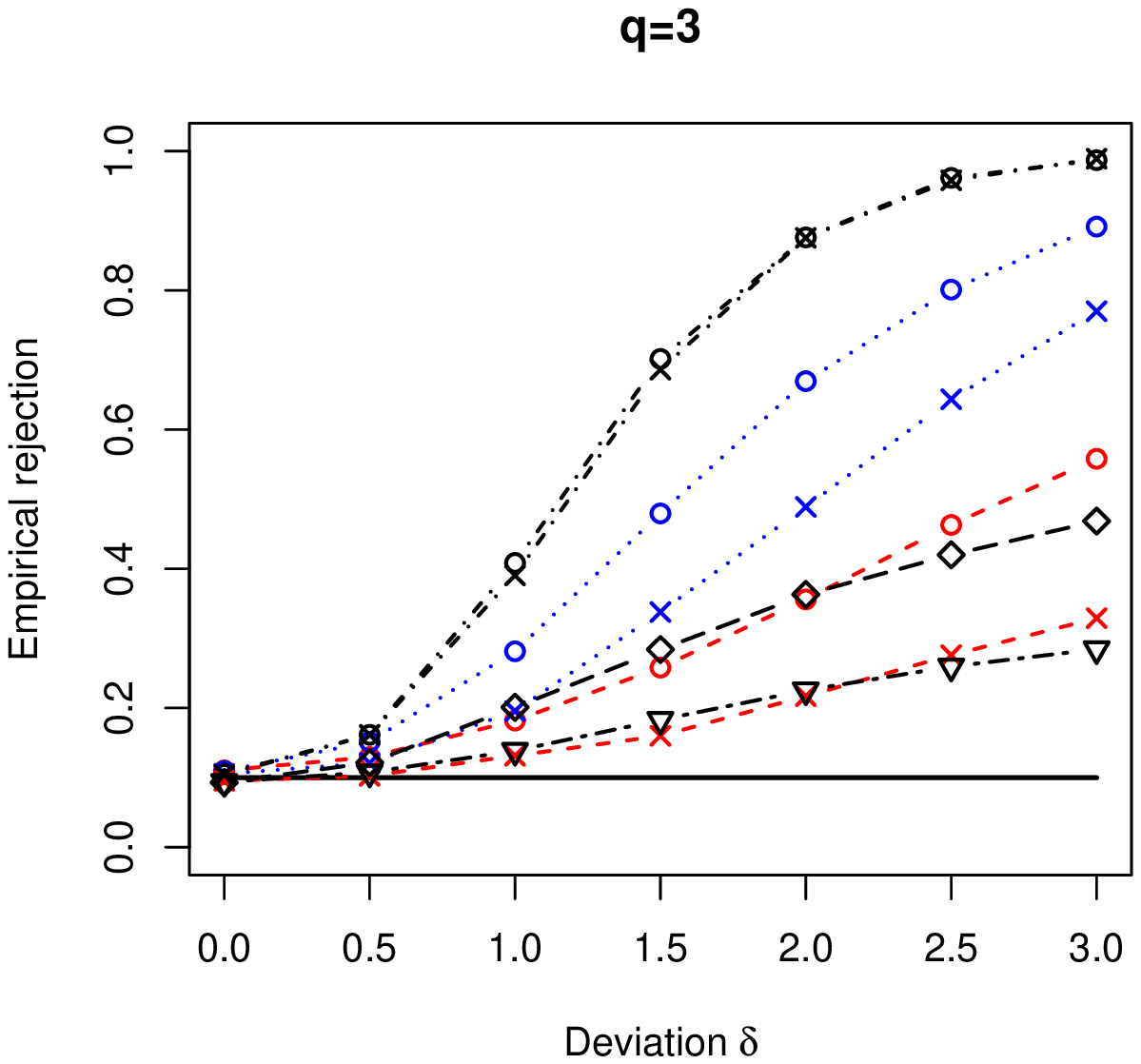}
 ~\!\!\includegraphics[scale=0.6,angle=270]{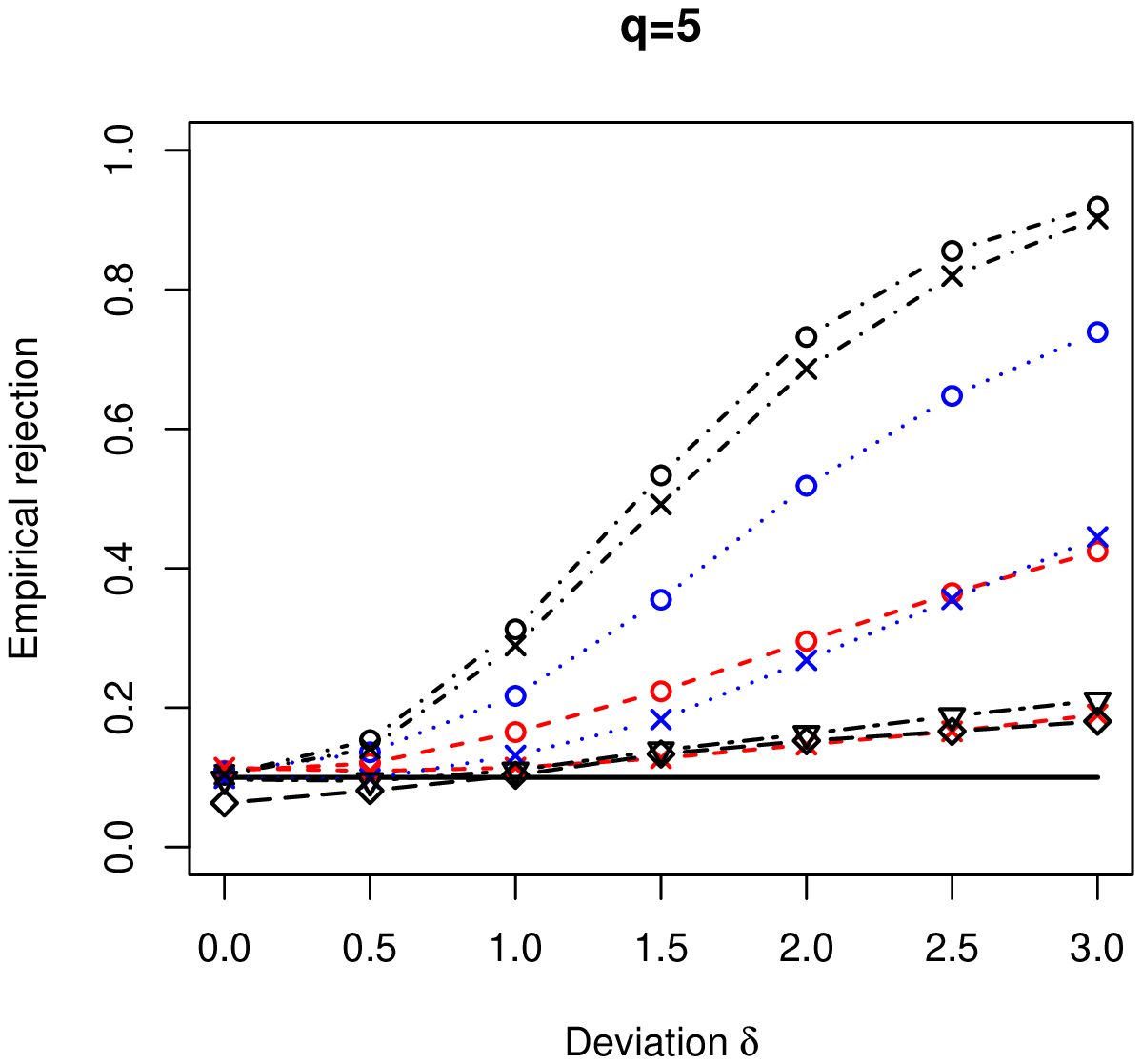}
 \tabularnewline
 ~\!\!\includegraphics[scale=0.6,angle=270]{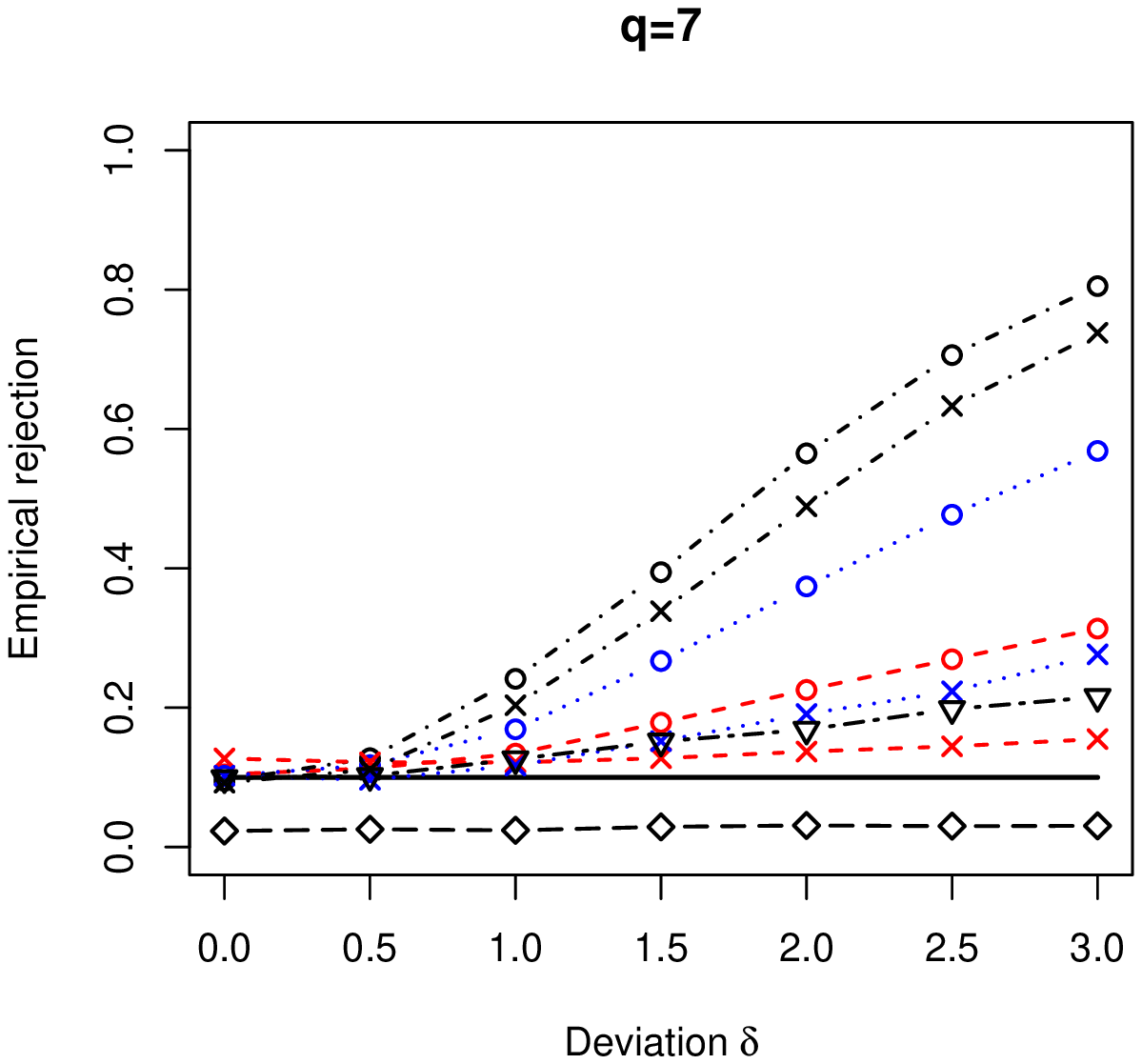}
 ~\!\!\includegraphics[scale=0.6,angle=270]{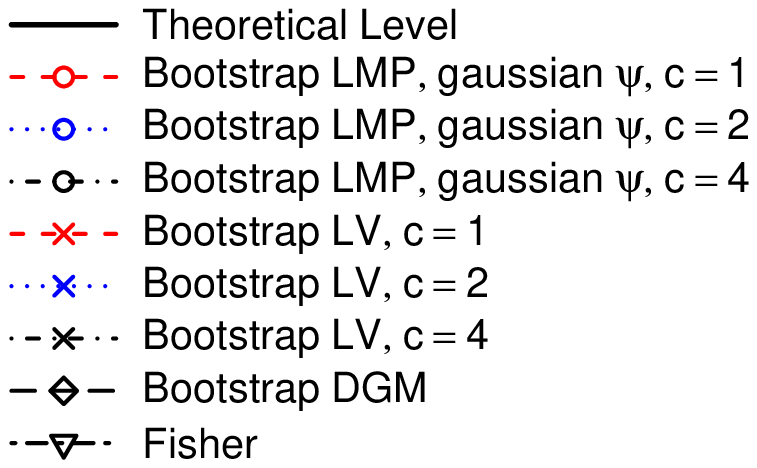}
 \tabularnewline
\end{tabular}
\par\end{centering}
\caption{Empirical power curves for a quadratic alternative, $n=100$ \label{fig:PowerContQs}}
\end{figure}

\begin{figure}
\begin{centering}
\begin{tabular}{c}
 \includegraphics[scale=0.6,angle=270] {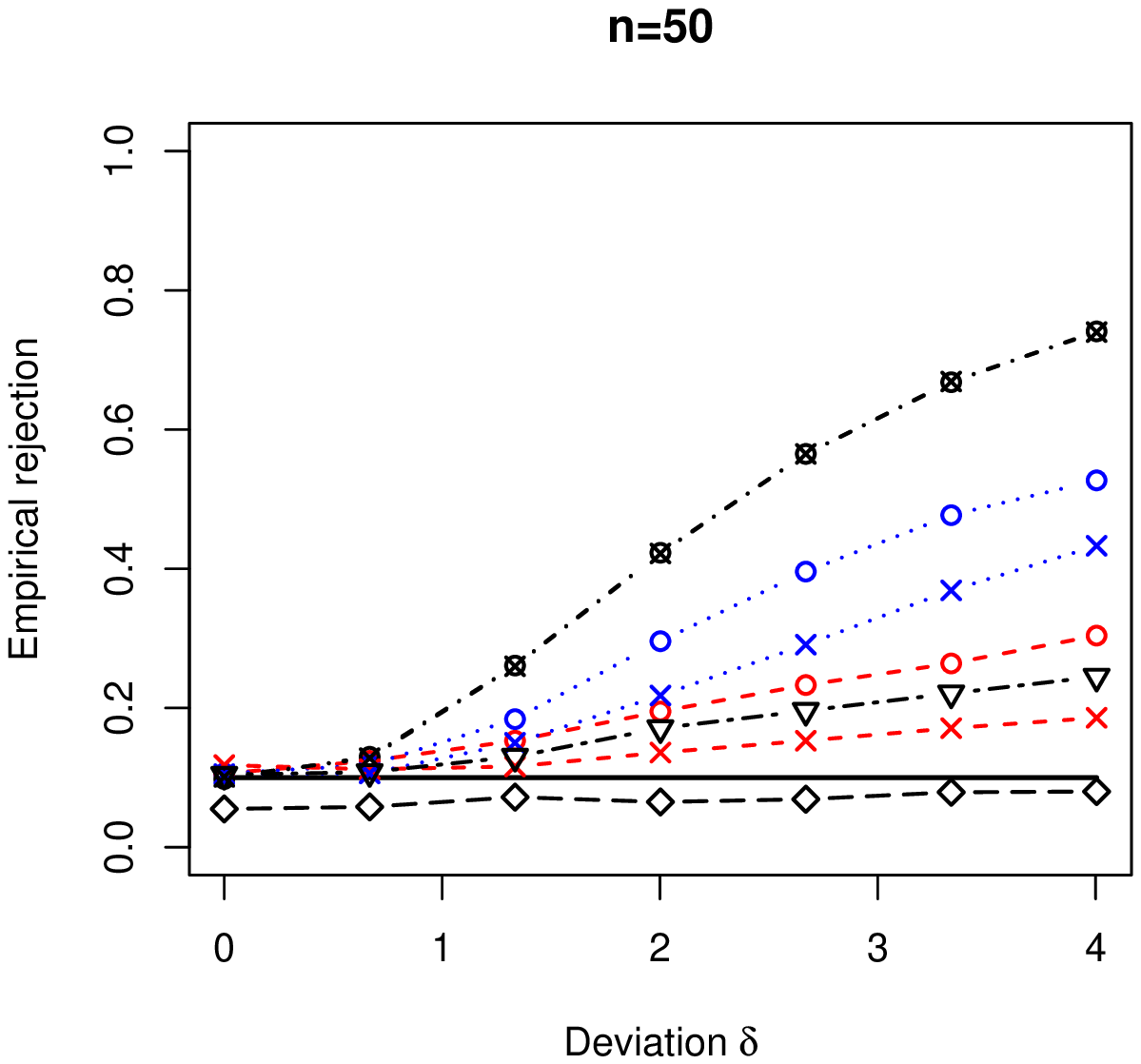}
 ~\!\!\includegraphics[scale=0.6,angle=270]{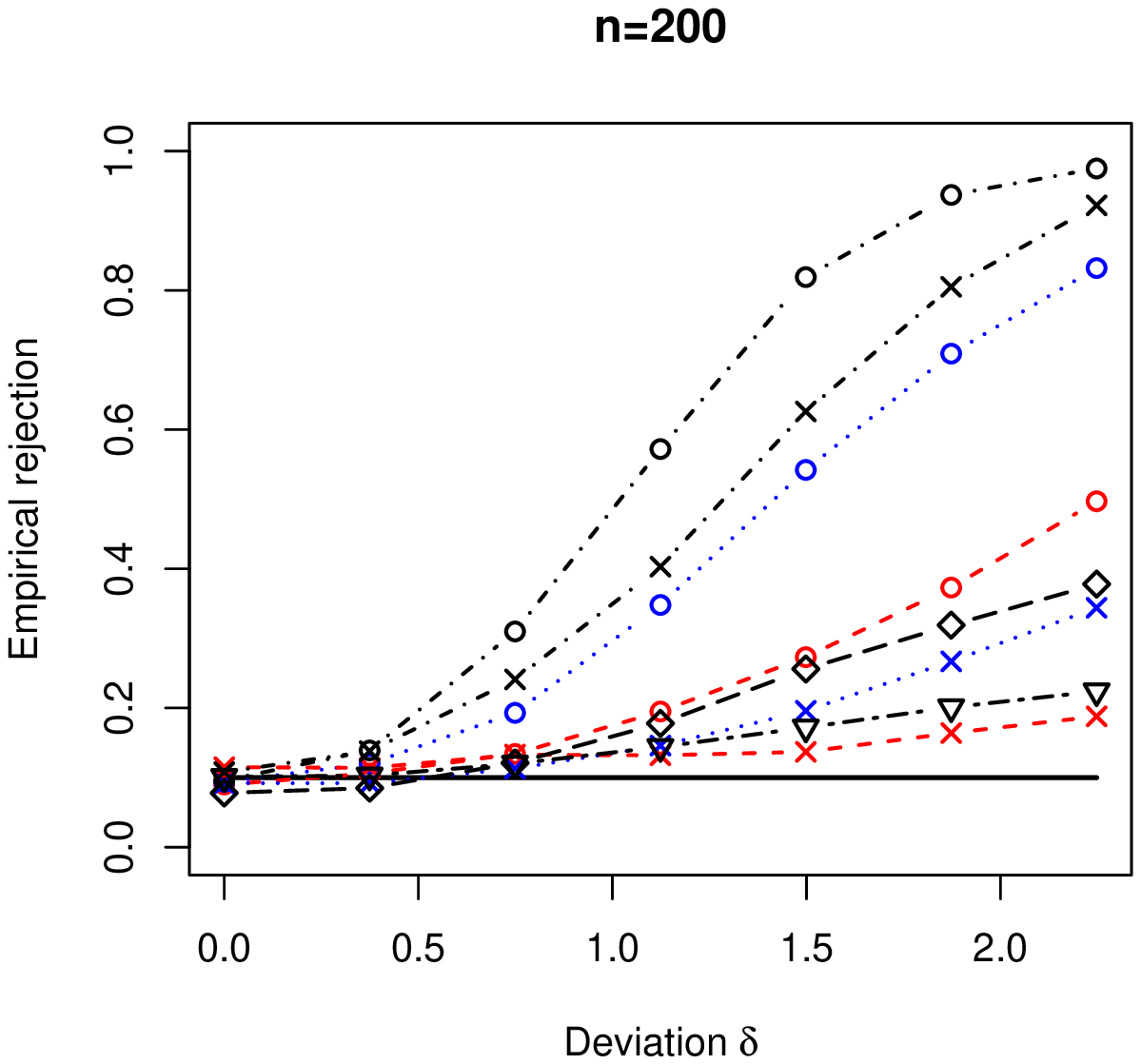}
 \tabularnewline
\end{tabular}
\par\end{centering}

\caption{Empirical power curves for a quadratic alternative, $q=5$ \label{fig:PowerContNs}}
\end{figure}

\begin{figure}
\begin{centering}
\begin{tabular}{c}
 \includegraphics[scale=0.6,angle=270]{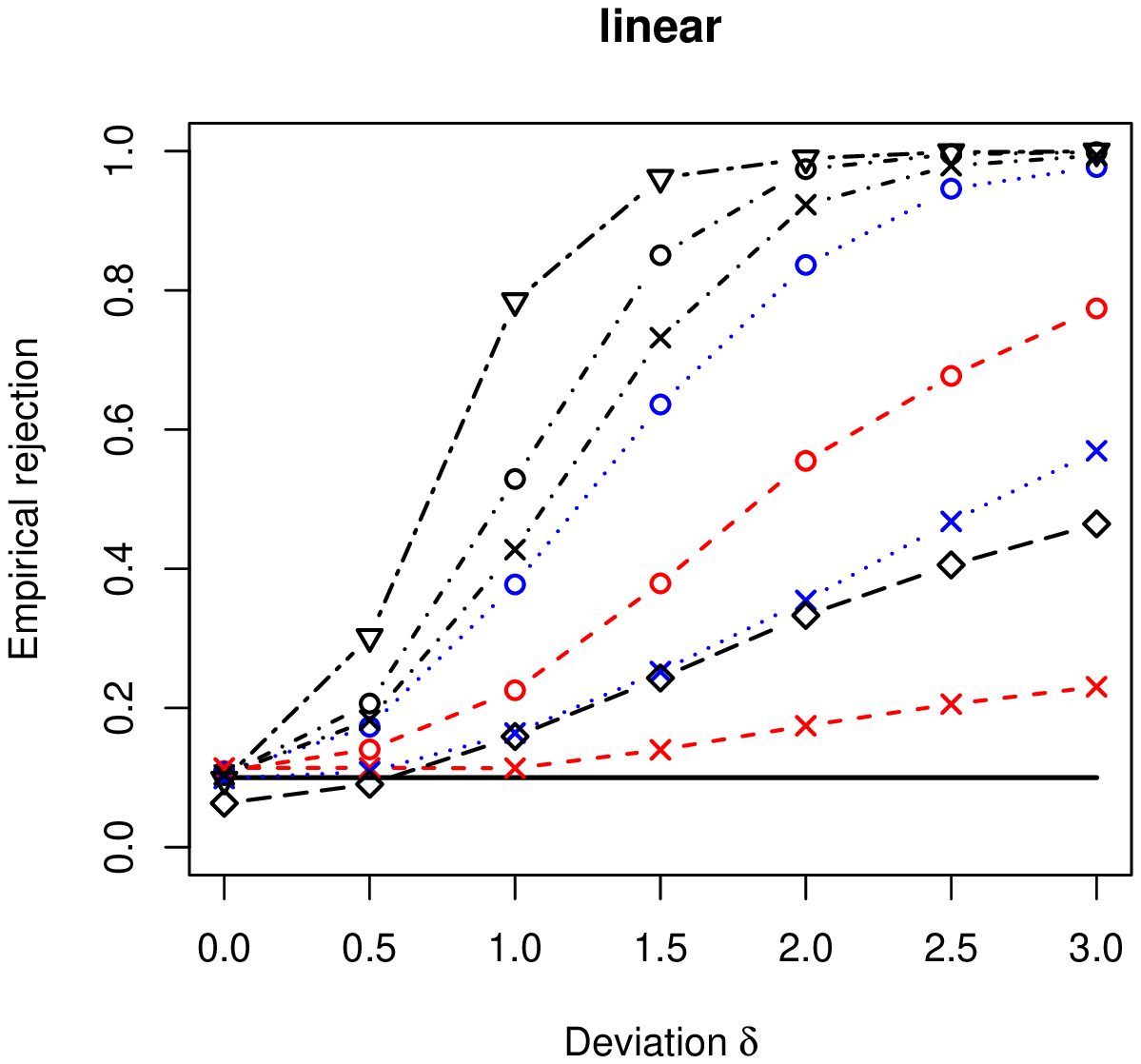} ~ \!\! \includegraphics[scale=0.6,angle=270]{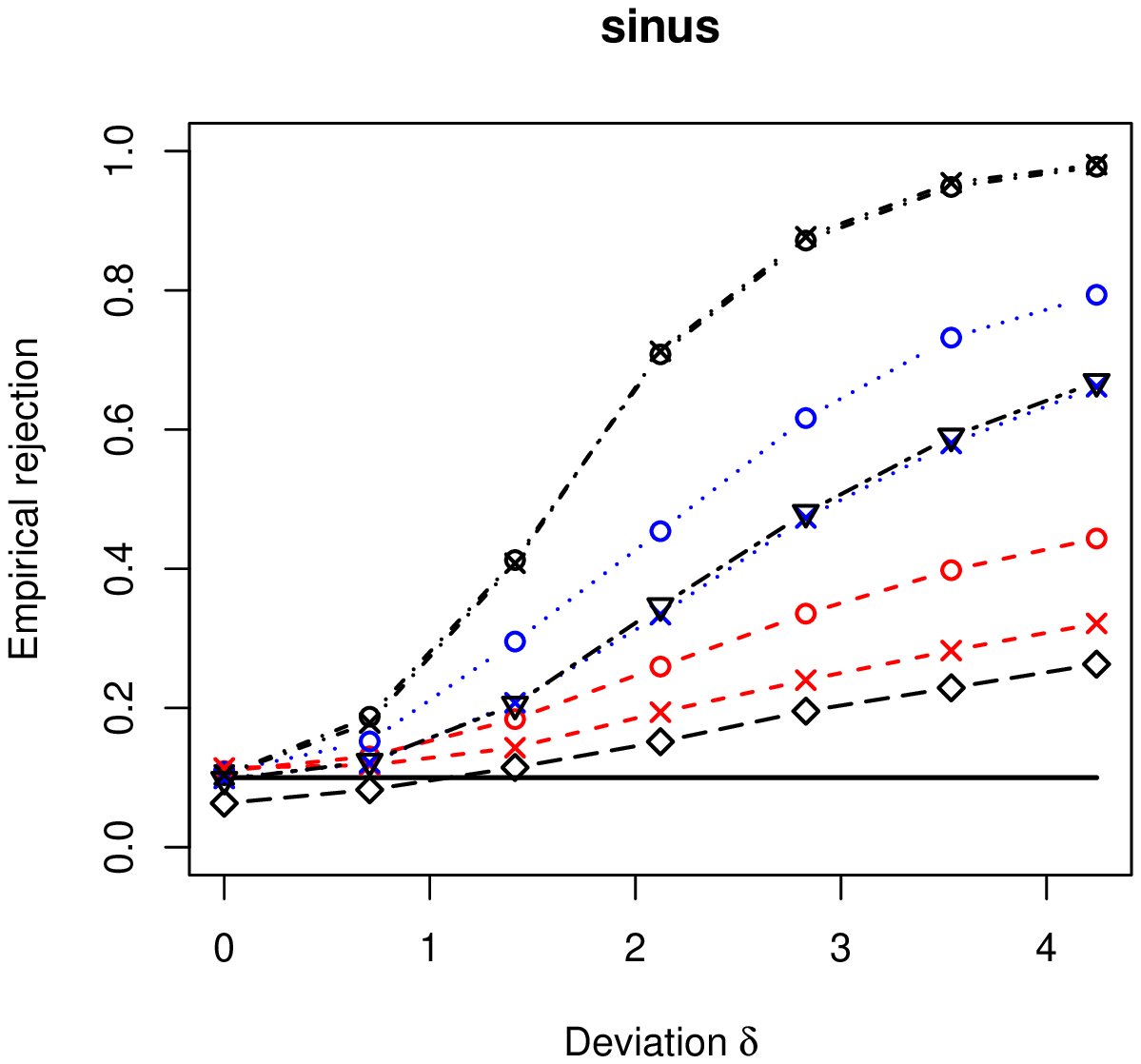}
\end{tabular}
\par\end{centering}
\caption{Empirical power curves for linear and sine alternative, $n=100$ and $q=5$}
\label{fig:PowerContAlter}
\end{figure}

\begin{figure}
\begin{centering}
\begin{tabular}{c}
 \includegraphics[scale=0.6,angle=270]{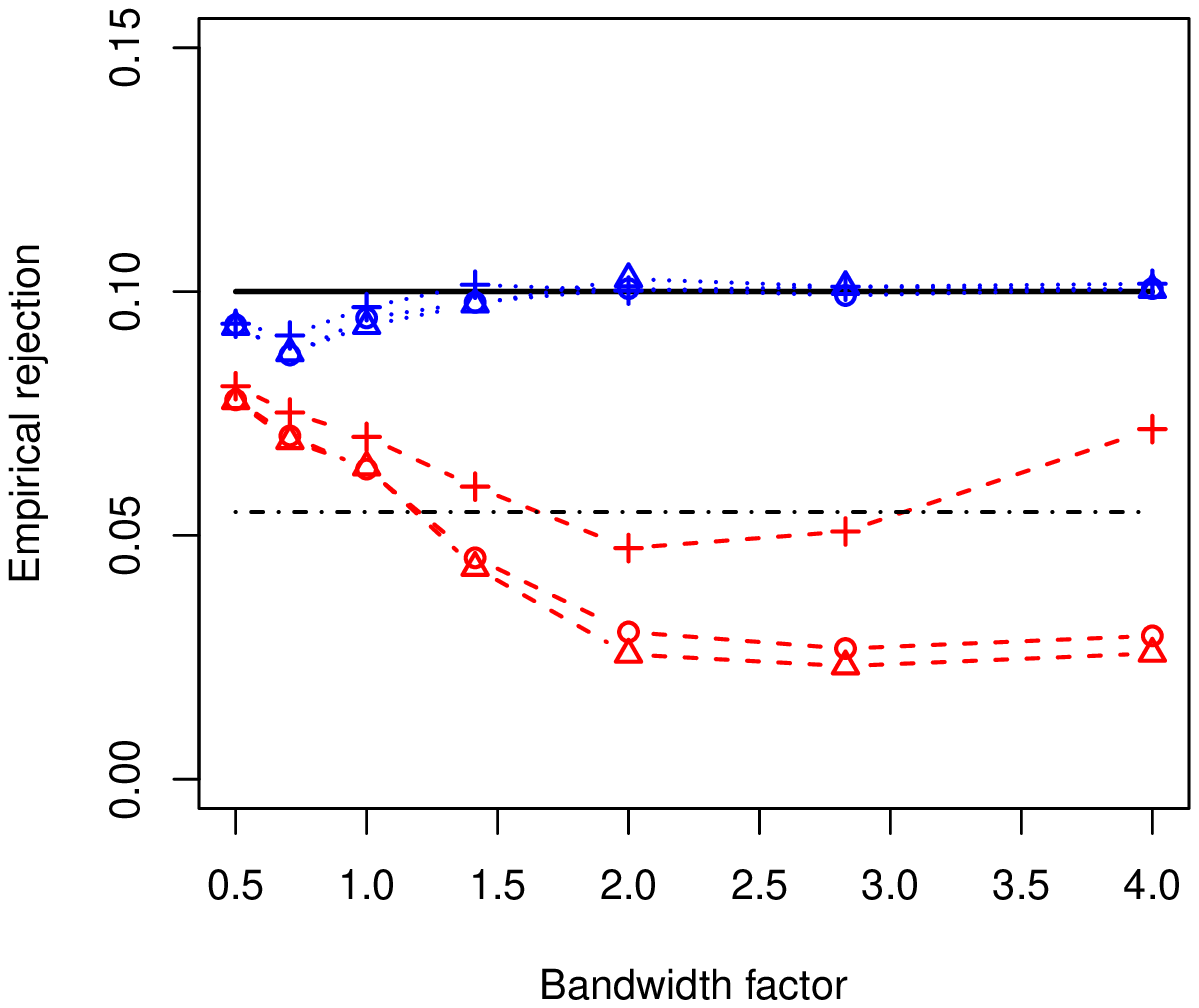} ~\!\! \includegraphics[scale=0.6,angle=270]{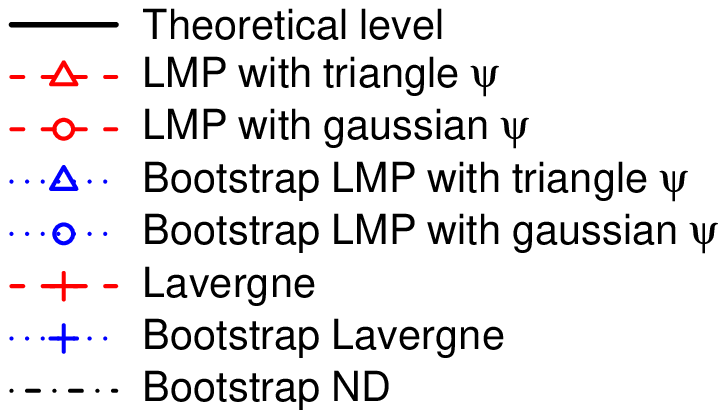}\tabularnewline
\end{tabular}
\par\end{centering}

\caption{Empirical rejection under $H_{0}$ as a function of the
  bandwidth, $X$ Bernoulli and $n=100$ \label{fig:LevelDisc}}
\end{figure}

\begin{figure}
\begin{centering}
\begin{tabular}{c}
 \includegraphics[scale=0.6,angle=270] {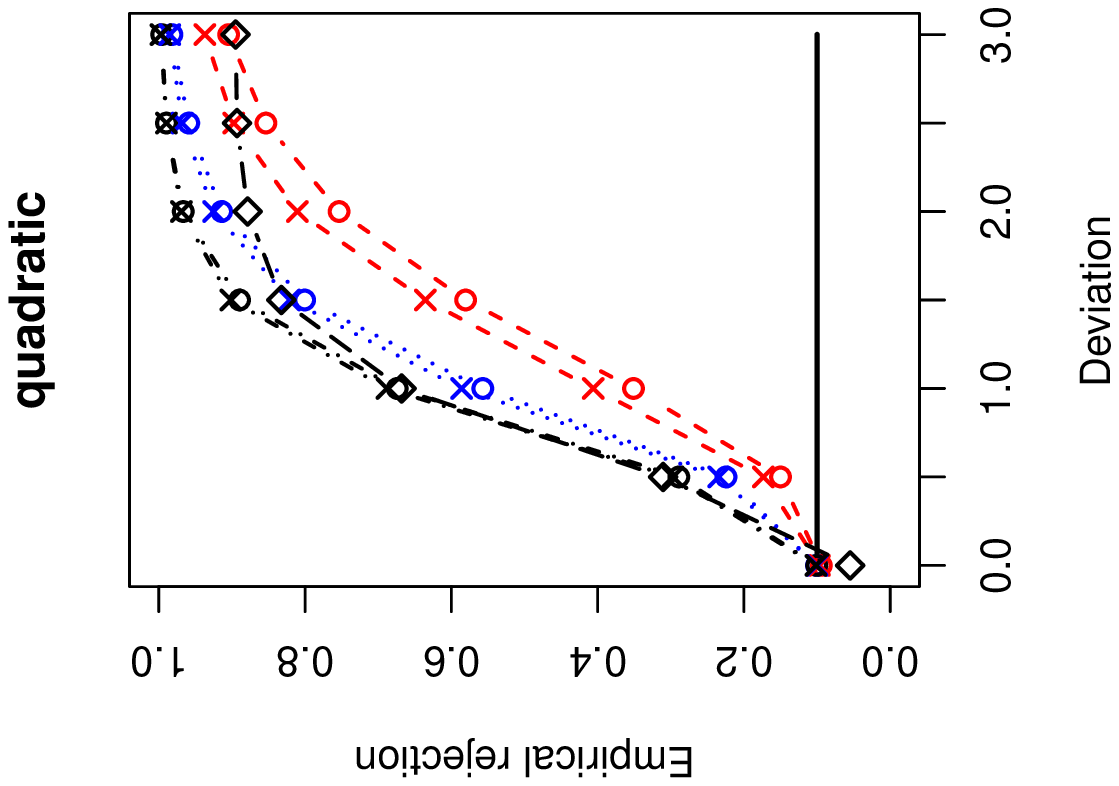}
 ~\!\!\includegraphics[scale=0.6,angle=270]{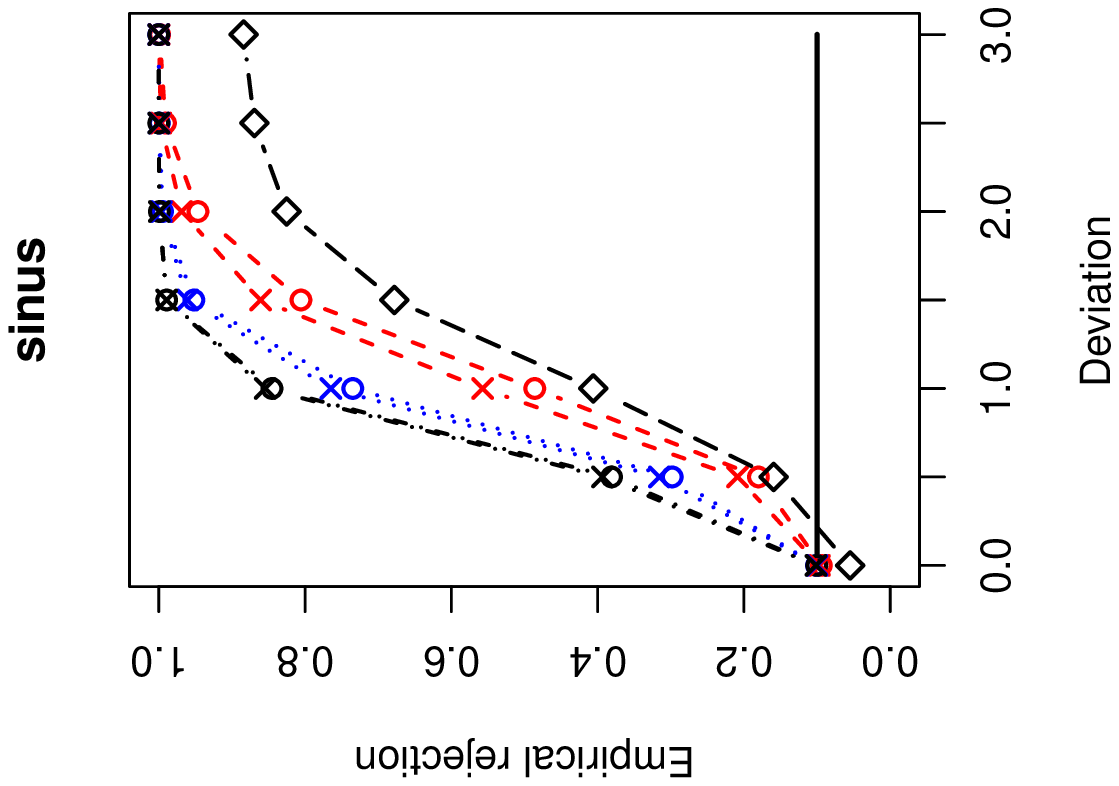}
 ~\!\!\includegraphics[scale=0.6,angle=270]{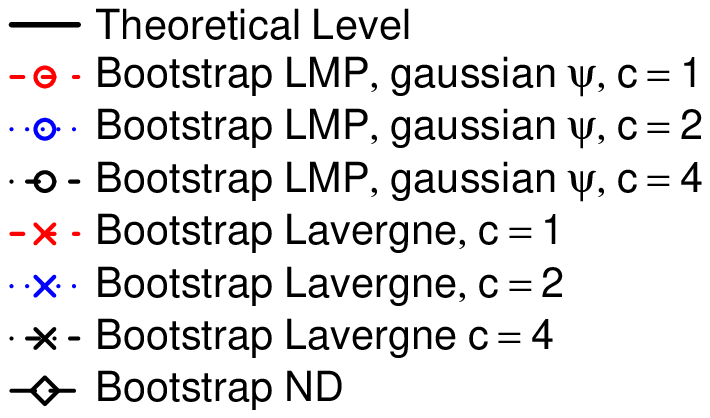}
 \tabularnewline
\end{tabular}
\par\end{centering}
\caption{Empirical power curves, $X$ Bernoulli  and $n=100$ \label{fig:PowerDisc}}
\end{figure}

\end{document}